\documentclass[12pt]{article}
\usepackage{amsmath,amssymb,amsfonts,amsthm}
\newcounter{item}[section]
\newcounter{kirshr}
\newcounter{kirsha}
\newcounter{kirshb}
\newenvironment{enumroman}{\setcounter{kirshr}{1}
\begin{list}{(\roman{kirshr})}{\usecounter{kirshr}} }{\end{list}}
\newenvironment{enumarab}{\setcounter{kirshb}{1}
\begin{list}{(\arabic{kirshb})}{\usecounter{kirshb}} }{\end{list}}
\newenvironment{athm}[1]{\vskip3mm\par\noindent
{\bf #1 }. \slshape }
{\upshape\par\vskip10pt minus3pt}
\newtheorem{theorem}{Theorem}[section]

\newtheorem{lemma}[theorem]{Lemma}
\newtheorem{corollary}[theorem]{Corollary}
\newenvironment{demo}[1]{\noindent{\bf #1.}\upshape\mdseries}
{\nopagebreak{\hfill\rule{2mm}{2mm}\nopagebreak}\par\normalfont}
\theoremstyle{definition}

\newtheorem{example}[theorem]{Example}
\newtheorem{definition}[theorem]{Definition}
\def\R{\mathbb{R}}

\def\C{{\mathfrak{C}}}

\def\Nr{{\mathfrak{Nr}}}

\def\RCA{{\sf RCA}}

\def\A{{\mathfrak{A}}}
\def\B{{\mathfrak{B}}}
\def\C{{\mathfrak{C}}}
\def\D{{\mathfrak{D}}}
\def\M{{\mathfrak{M}}}
\def\N{{\mathfrak{N}}}

\def\Bl{{\mathfrak{Bl}}}
\def\CA{{\sf CA}}

\def\SC{{\bf SC}}
\def\QEA{{\bf QEA}}

\def\Uf{{\sf Uf}}

\def\K{{\bf K}}
\def\K{{\bf K}}

\def\RCA{{\sf RCA}}

\def\Rd{{\ Rd}}
\def\(R)RA{{\bf (R)RA}}
\def\RA{{\bf RA}}
\def\RRA{{\bf RRA}}

\def\R{\mathbb{R}}
\def\F{{\sf F}}

\def\Rl{{\sf Rl}}
\def\c #1{{\cal #1}}
 \def\CA{{\sf CA}}
\def\B{{\sf B}}
\def\G{{\sf G}}
\def\w{{\sf w}}
\def\y{{\sf y}}
\def\g{{\sf g}}
\def\b{{\sf b}}
\def\r{{\sf r}}
\def\K{{\sf K}}

 \def\Cm{{\mathfrak{Cm}}}
\def\Nr{{\mathfrak{Nr}}}

\def\restr #1{{\restriction_{#1}}}
\def\cyl#1{{\sf c}_{#1}}

\def\Ra{{\mathfrak{Ra}}}

\def\set#1{\{#1\} }
\def\Ra{{\mathfrak{Ra}}}
\def\Nr{{\mathfrak{Nr}}}
\def\Tm{{\mathfrak{Tm}}}
\def\A{{\mathfrak{A}}}
\def\B{{\mathfrak{B}}}
\def\C{{\mathfrak{C}}}
\def\D{{\mathfrak{D}}}

\def\A{{\mathfrak{A}}}
\def\B{{\mathfrak{B}}}
\def\C{{\mathfrak{C}}}
\def\D{{\mathfrak{D}}}

\def\P{{\mathfrak{P}}}

\def\Bb{{\mathfrak{Bb}}}
\def\L{{\mathfrak{L}}}
\def\PEA{{\mathbf{PEA}}}
\def\PA{{\mathbf{PA}}}
\def\Bb{{\mathfrak{Bb}}}
\def\L{{\mathfrak{L}}}
\def\CA{{\sf CA}}
\def\RA{{\sf RA}}
\def\RRA{{\sf RRA}}
\def\RCA{{\sf RCA}}

\def\G{{\bf G}}

\def\CRA{{\bf CRA}}

\def\At{{\sf At}}

\def\Ra{{\sf Ra}}
\def\R{{\sf R}}
\def\Rd{{\sf Rd}}
\def\nodes{{\sf nodes}}

\def\rng{{\sf rng}}
\def\dom{{\sf dom}}
\def\Z{{\sf Z}}
\def\pa{$\forall$}

\def\ws{winning strategy}

\def\y{{\sf y}}
\def\g{{\sf g}}
\def\bb{{\sf b}}
\def\r{{\sf r}}
\def\w{{\sf w}}
\def\PA{{\sf PA}}
\def\PEA{{\sf PEA}}
 
\def\QEA{{\sf QEA}}
\def\SC{{\sf SC}}
\def\CRA{{\sf CRA}}
\def\Cof{{\sf Cof}}
\def\pe{$\exists$}

\title{Various interplays between relation and cylindric algebras}
\author{Tarek Sayed Ahmed }
\begin{document}
\maketitle

\begin{abstract} Using model theoretic techniques that proved that the class of $n$ neat reducts of $m$ dimensional cylindric algebras, $\Nr_n\CA_m$,
is not elementary, we prove the same result for $\Ra\CA_k$, $k\geq 5$, and we show that $\Ra\CA_k\subset S_c\Ra\CA_k$ for all $k\geq 5$. 
Conversely, using the 
rainbow construction for cylindric algebra, 
we show that several classes of algebras, related to the class $\Nr_n\CA_m$, 
$n$ finite and $m$ arbitrary, are  not elementary. Our results apply to many cylindric-like algebras, 
including Pinter's substitution algebras and Halmos' polyadic 
algebras with and without 
equality. The techniques used are essentially those used by Hirsch and Hodkinson, and later by Hirsch in \cite{hh} and \cite{r}.
In fact, the main result in \cite{hh} follows from our more general construction.
Finally we blow up a little the {\it blow up and blur construction} of Andr\'eka nd N\'emeti, 
showing that various constructions of weakly representable atom structures
that are not strongly representable, can be formalized in our blown up, blow up and blur construction, 
both for relation and cylindric algebras. Two open problems are discussed, in some detail, proposing ideas.
One is whether class of subneat reducts are closed under completions, the other is whether there exists a 
weakly representable $\omega$ dimensional atom structure,
that is not strongly representable. For the latter we propose a lifting argument, due to Monk, 
applied to what we call anti-Monk algebras (the algebras, constructed
by Hirsch and Hodkinson, are atomic, and their atom structure is stongly representable.)
\end{abstract}

\section*{Introduction}

Relation algebras and cylindric algebras introduced by Tarski are cousins. The concrete
version of relation algebras are algebras of binary relations, with the binary operation of composition and unary one of 
taking converses,
while the concrete version of cylindric algebras of dimension $n$ are algebras of $n$ ary relations
with $n$ unary operations of cylindrfiers or projections and constants, the diagonal elements
reflecting equality.

There is an endless interplay between relation algebras and cylindric algebras. One can construct a relation algebra from a $\CA_n$,
and conversely from an $\RA$, that has what Maddux calls an $n$ dimensional cylindric basis, one can construct a $\CA_n$, 
and even without having a basis, as shown by Hodkinson,  though in the latter construction one loose that original relation algebra is embeddable into the 
$\Ra$ reduct of the new $\CA_n$.

They have a lot of common features and manifestations. For example the representable algebras (in both cases, for $\CA_n$ $n\geq 3$)
are unruly and wild, being resilient to any simple axiomatizations; it is undecidable to tell whether a finite algebra is representable or not.
It is an entertaining practice among algebraic logicians to transfer results from $\RA$ to $\CA$ and vice versa,  
but  for a statement that applies to both, it is usually easier to prove it for $\RA$s (like the non finite-axiomatizability of the 
equational axiomatizations of the representables and the undecidablity problem for finite algebras, as to wether they are representable
or not). Indeed, this has been mostly the case historically in a temporal sense, 
but of course there are exceptions, for example  the class of neat reduct 
was proved non-elementary before
the class of $\Ra$ reducts \cite{r}. However, even in this case, the $\CA$ analogue of some results proved in the latter reference 
will only be proved here.

There are also types  of construcions that apply to both.
In this paper we will be concerned with the rainbow construction, 
to prove new results concerning various subclasses that are related to neat rducts one way 
or another, and the so called Blow up and Blur construction, a construction invented by Andr\'eka and N\'emeti to 
show that there are weakly representable atom structures that are not strongly
representable; not only that, but the term algebra can be a $k$- neat reduct for any pre assigned finite $k$.
(It cannot be in $\Nr_n\CA_{\omega}$ for in this case, having only countable many atoms it will be completely representable, forcing
the complex algebra to be representable as well).


We start by give a unified model theoretic proof to several deep results that are scattered
in the literature in a serious of publications dating back to the seventees of the last century.
We follow the notation of \cite{HMT1} and \cite{1}. In particular, for ordinals $\alpha<\beta$, $\Nr_{\alpha}\CA_{\beta}$ 
denotes the class of  $\alpha$ neat reducts 
of $\CA_{\beta}$s.
Henkin et all \cite{HMT1} showed that $\Nr_1\CA_{\beta}$ is a variety for any $\beta\geq 1$.
N\'emeti \cite{Nemeti} showed that for any pair of ordinals $1<\alpha,\beta$, the class $\Nr_{\alpha}\CA_{\beta}$, though closed under 
homomorphic images and products, 
is not closed under forming subalgebras answering problem 2.11 in\cite{HMT1}, hence is not a variety.
N\'emeti then posed the question as to whether
it is closed under {\it elementary} subalgebras. Being closed under ultraproducts, as proved by Henkin et all \cite{HMT1},
this amounts to asking whether it is elementary or not. 
Andr\'eka and N\'emeti proved that is for the  lowest dimension $2$, namely $\Nr_2\CA_{\beta}$ is not elementary for any $\beta>2$.

The result was extended to all dimension by the author, and a model theoretic proof was given in \cite{MLQ}; the same 
proof reported in some detail in the survey article \cite{Sayedneat}. 
This is the method we use here.

Later the problem was investigated for cylindric like algebras, like substitution algebras $\SC$ of Pinter, and polyadic 
algebras of Halmos $\PA$.
The problem was solved for any class $\K$ between $\SC_2$ and $\PEA_2$ in \cite{FM}; 
the proof hence generalizes that
of Andr\'eka and N\'emeti. This result  was extended to finite dimensions in \cite{Studia}, and infinite dimensions in [\cite{quasipolyadic}.

In the context of cylindric algebras,  closure under complete neat embeddings and complete representability
was proved equivalent for countable atomic algebras by the author \cite{Sayed}
The charcterization also works for relation algebras, using the same method, which is a Baire category argument at heart
\cite{1}; later reproved by Robin Hirsch using games \cite{r}.  It was also proved that all three conditions cannot be omitted, 
atomicity, countability and complete embeddings.
There are examples, that show that such conditions are necessary.

Hirsch and Hodkinson \cite{HHbook} prove
that the class of completely representable $\CA_n$s is not elementary, for any $n\geq 3.$

For our results concerning neat reducts, we 
use techniques of Hirsch's in \cite{r} that deal with relation algebras,  and those of Hirsch and Hodkinson in \cite{hh} 
on complete representations. The results in the latter 
had to do with investigating the existence of complete representations for cylindric algebras  and for
this purpose, an infinite (atomic) game that tests complete representability was devised, and such a game was used
on a rainbow relation algebra.
The rainbow construction has a very wide scope of applications, 
and it  proved to be a nut cracker in solving many hard problems for relation algebras,
particularly for constructing counterexamples distinguishing between classes that are very subtly related, or rather unrelated. 

Unfortunately, relation rainbow algebras do not posses cylindric basis for $n\geq 4$ 
(so it seems that we cannot have our cake and eat it), so to
prove the analogous  result for cylindric algebra the construction had to be considerably modified to adapt the new situution, starting anew,
though the essence of
the two constructions is basically the same.
Instead of using atomic networks, in the cylindric algebra case games are played on {\it coloured graphs}. 
On the one hand, such graphs have edges which code the
relation algebra construction, but they also have hyperdges of length $n-1$, reflecing the cylindric algebra structure. 

It seems that there is no general theorem for rainbow constructions when it comes to cylindric like algebras, 
namely one relating winning strategies for pebble games on two structures or graphs $A,B$, to winning strategies for \pe\ in the cylindric rainbow 
algebra based $A$ and $B$, \cite{HHbook2}. 

Nevertheless, 
in the latebook on cylindric algebra \cite{1}, a general rainbow construction
is given in \cite{HHbook} in the context of building algebras from graphs giving rise to a class of models, 
from which the rainbow atom structure is defined, but just referring to one graph as a parameter, rather than two structures as done in their 
earlier book \cite{HHbook2}.
The second graph is fixed to be the greens; these are the set of colours that \pe\ never uses.
(In our class the class of models will be coloured graphs, viewed as structures for a natural signature).

For cylindric algebras, we take the $n$ neat reducts of algebras in higher dimension, ending up with a $\CA_n$, 
but we can also take {\it relation algebra reducts}, getting instead a relation algebra. 
The class of relation algebra reducts of cylindric algebras of dimension $n\geq 3$, denoted 
by $\Ra\CA_n$. The $\Ra$ reduct of a $\CA_n$, $\A$, is obtained by taking the $2$ neat reduct of $\A$, then defining composition and converse 
using one space dimension.
For $n\geq 4$, $\Ra\CA_n\subseteq \RA$. Robin Hirsch dealt primarily with this class in \cite{r}.   
This class has also been investigated by many authors, 
like Monk, Maddux, N\'emeti and Simon (A chapter in Simon's dissertation is devoted to such a class, 
when $n=3$). 
After a list of results and publications, Simon proved $\Ra\CA_3$ is not closed under subalgebras for $n=3$, 
with a persucor by Maddux proving the cases $n\geq 5$,
and Monk proving the case $n=4$.

In \cite{r}, Hirsch deals only the relation algebras proving that the $\Ra$ reducts of $\CA_k$s, $k\geq 5$, 
is not elementary, and he ignored the $\CA$ case, probably
because  of analogous results proved by the author on neat reducts \cite{IGPL}.

But  the results in these two last  papers are not identical (via a replacement of relation algebra via a cylindric algebra and vice versa).
There are differences and similarities that are illuminating 
for both. For example in the $\RA$  case Hirsch proved that the elementary subalgebra that is not an $\Ra$ reduct 
is not a complete subalgebra of the one that is.
In the cylindric algebra case, the elementary subalgebra that is not a neat reduct 
constructed is a complete subalgebra of the neat reduct.

Hirsch \cite{r} also proved that any $\K$, such that $\Ra\CA_{\omega}\subseteq \K\subseteq S_c\Ra\CA_k$, $k\geq 5$ 
is not elementary;  here, using a rainbow construction for cylindric algebras, we prove its $\CA$ analogue.
In the same paper \cite{r}. In op.cit Robin asks whether the inclusion $\Ra\CA_n\subseteq S_c\Ra\CA_n$ is proper, the construction 
in \cite{IGPL}, shows that for $n$ neat reducts, 
it is. 

Besides giving a unified  proof of all cylindric like algebras for finite dimensions, 
we show that the inclusion is proper given that a certain $\CA_n$ term exists. 
(This is a usual first order formula using $n$ variables).
And indeed  using the technique in \cite{IGPL} we prove an analogous result for relation algebras,
answering the above  question of Hirsch's in \cite{r}. We show that there is an $\A\in \Ra\CA_{\omega}$ with a  an elmentary subalgebra
$\B\in S_c\Ra\CA_{\omega}$, that is not in $\Ra\CA_k$ when $\leq 5$. In particular, $\Ra\CA_k\subseteq S_c\Ra\CA_5$, for
$k\geq 5$.

Now it is worthwhile to reverse the deed, and generalize Hirsch's construction using rainbow cylindric algebras, to more results than that
obtained for cylindric algebras on neat reducts in \cite{IGPL}. 
For example, our construction here will give the following result not proved in {\it op.cit}:
There is an algebra $\A\in \Nr_n\CA_{\omega}$ with an elementary subalgebra, that is not completely
representable. But since the algebra $\A$ has countably many atoms, then it is completely representable.
This gives the result in \cite{hh}.

The transfer from results on relation algebras to cylindric algebras is not mechanical at all. More often than not, this is not an easy task, 
indeed it is far from being  trivial. 

We use essentially the techniques in \cite{r}, together with those in \cite{hh}, extending the rainbow construction
to cylindric algebra. But we mention a very important difference. 

In \cite{hh} one game is used to test complete representability.
In  \cite{r} {\it three} games were divised testing different neat embedability properties.
(An equivalence between complete representability and special neat embeddings is proved in \cite{IGPL})

Here we use only two games adapted to the $\CA$ case. This suffices for our purposes. 
The main result in \cite{hh}, namely, that the class of completely representable algebras of dimension
$n\geq 3$, is non elementary, follows from the fact that \pe\  cannot win the infinite length 
game, but he can win the finite ones. 

Indeed a very useful way of characterizing non elementary classes, 
say $\K$, is a Koning lemma argument. The idea is  to devise a game $G$ on atom structures such that for a given algebra atomic $\A$  \pe\ 
has a winning strategy on its atom structure for all games of finite length, 
but \pa\ wins the $\omega$ round game. It will follow that there a countable cylindric algebra $\A'$ such that $\A'\equiv\c
A$ and \pe\ has a \ws\ in $G(\A')$.
So $\A'\in K$.  But $\A\not\in K$
and $\A\preceq\A'$. Thus $K$ is not elementary.

To obtain our result we use {\it two} distinct games, both having $\omega$ rounds.  
Of course the games are very much related.

In this new context \pe\ can also win a finite game with $k$ rounds for every $k$. Here the game
used  is more complicated than that used in Hirsch and Hodkinson, 
because in the former case we have three kinds of moves which makes it harder for \pe\ 
to win. 

Another difference is that the second game, call it $H$,  is actually  played on pairs, 
the first component is an atomic network (or coloured graph)  defined in the new context of cylindric 
algebras, the second is a set of hyperlabels, the finite sequences of nodes are lablled, 
some special ones are called short, and {\it neat} hypernetworks or hypergraphs are those that label short hyperedges with the same label. 
And indeed a \ws\ for \pe\ 
in the infinite games played on an atom structure forces that this is the atom structure of a neat reduct; 
in fact an algebra in $\Nr_n\CA_{\omega}$. However, unlike complete representability,
does not exclude the fact, in principal, there are other representable algebras having the same atom stucture can be only subneat reducts.

On the other hand, \pa\  can win 
{\it another  pebble game}, also in $\omega$ rounds (like in \cite{hh} on a red clique), but
 there is a  finiteness condition involved in the latter, namely  is the number of nodes 'pebbles 'used, which is $k\geq n+2$, 
and  \pa\ s \ws\ excludes the neat embeddablity of the algebra in $k$ extra dimensions. This game will be denoted by $F^k$. 

And in fact the Hirsch Hodkinson's main result in \cite{r}, 
can be seen as a special case, of our construction. The game $F^k$, without the restriction on number
of pebbles used and possibly reused, namely $k$ (they have to be reused when $k$ is finite), but relaxing the condition of finitness,
\pa\ does not have to resuse node, and then this game  is identical to the game $H$ when we delete the  hyperlabels from the latter, 
and forget about the second and third 
kinds of move. So to test only complete representability, we use only these latter games, which become one, namely the one used 
by Hirsch and Hodkinson in \cite{hh}. 
In particular, our algebra $\A$ constructed below is not completely representable, but is elementary equivalent to one that is.
This also implies that the clas of completely representable atom structures are not elementary, the atom structure of the 
former two structures are elementary equivalent, one is completely representable, the other is not.
Since an atom structure of an algebra is first order interpretable in the algebra, hence, 
the latter also gives an example of an atom structure that is weakly representable but not strongly representable, showing that
the class $\CRA_n$ is not elementary.

Concerning the blow up and blur construction, we give a simpler proof of a result of Hodkinson as an instance 
of such (blur and blow) constructions
argueing that the idea at heart is similar to that adopted by Andr\'eka et all \cite{sayed}.
The idea is to blow up a finite structure, replacing each 'colour or atom' by infinitely many, using blurs
to  represent the resulting term algebra, but the blurs are not enough to blur the structure of the finite structure in the complex algebra. 
Then, the latter cannot be representable due to a {\it finite- infinite} contradiction. 
This structure can be a finite clique in a graph or a finite relation algebra or a finite 
cylindric algebra. This theme gives example of weakly representable atom structures that are not strongly
representable. This is the essence too of construction of Monk like-algebras, one constructs graphs with finite colouring (finitely many blurs),
converging to one with infinitely many, so that the original algebra is also blurred at the complex algebra level, 
and the term algebra is completey representable, yielding a representation of its completion the 
complex algebra. 

A reverse of this process exists in the literature; 
that can be called anti-Monk algebras, it builds algebras with infinite blurs converging to one with finite blurs. This idea due to 
Hirsch and Hodkinson, uses probabilistic methods of Erdos to construct a  sequence of graphs with infinite  chromatic 
number one that is $2$ colourable. This construction, which works for both relation and cylindric algebras,
further shows that the class of strongly representable atom structures
is not elementary. Using such algebras we will give an idea of how to construct weakly representable infinite dimensional atom structure
that is not strongly representable.

{\bf Layout}  In section 1, we prove the $\Ra$ analogue of the results  on neat reducts using model theoretic techniques in \cite{IGPL}. In section 2,
we prove the $\CA$ analogue of the results in \cite{r} to cylindric algebras, using the rainbow construction. In the last section, we give a general form of 
the so called Blow up and Blur construction, a term and construction invented by Andr\'eka and N\'emeti, and we present many constructions in the literature proving the exsistence of weakly representable
algebras that are not strongly representable, as an instance of our general framework.

\section{The class $\Ra\CA_n$}

Here we manifest yet another interplay between relation algebras and cylindric algebras.
Using a construction for cylindric algebra given in \cite{MLQ}, and given in some detail in \cite{Sayedneat}, we prove a result on
relation algebras, and reprove several results for cousins of cylindric algebras.
Our model-theoretic proof, unifies results and proofs in \cite{FM}, \cite{Studialogica}, \cite{IGPL}, using the methods in \cite{MLQ}.
The methods used in the former three references are more basic. The advantage in the method used in \cite{MLQ}, 
is that it uses rather sophisticated methods in Model theory, namely, 
Fraisse's methods of constructing homogeneous models that admit elimination of quantifiers, by amalgamating its smaller
parts.

We will not give the complete proof in detail; the proof is complete modulo the existence of
two terms, a cylindric algebra term, and a relation algebra term,
whose definition we omit, but their properties will be clearly stated; and we
hope that the general idea will be clear modulo this omission.

\subsection{Neat and $\Ra$ reducts of cylindric algebras}

We shall prove (the second item (modulo the existence of a $k$ witness)  answers a question of Hirsch \cite{r}.)

\begin{theorem} Let $\K$ be any of cylindric algebra, polyadic algebra, with and without equality, or Pinter's substitution algebra.
We give a unified model theoretic construction, to show the following:
\begin{enumarab}
\item For $n\geq 3$ and $m\geq 3$, $\Nr_n\K_m$ is not elementary, and $S_c\Nr_n\K_{\omega}\nsubseteq \Nr_n\K_m.$
\item Assume that there exists a $k$-witness. For any $k\geq 5$, $\Ra\CA_k$ is not elementary
and $S_c\Ra\CA_{\omega}\nsubseteq \Ra\CA_k$.
\end{enumarab}
\end{theorem}


A $k$ witness which is a $\CA_k$ term with special properties will be defined below. 
For $\CA$ and its relatives the idea is very much like that in \cite{MLQ}, the details implemented, in each separate case,
though are significantly distinct, because we look for terms not in the clone of operations
of the algebras considered; and as much as possible, we want these to use very little spare dimensions, hopefuly just one.

The relation algebra part is more delicate. We shall construct a relation algebra $\A\in \Ra\CA_{\omega}$ with a complete subalgebra $\B$,
such that $\B\notin \Ra\CA_k$, and $\B$ is elementary equivalent to $\A.$ (In fact, $\B$ will be an elementary subalgebra of $\A$.)

We work with $n=3$. One reason, is that for higher dimensions the proof is the same.
Another one is that in the relation algebra case, we do not need more dimensions.

Roughly, the idea is to use an uncountable cylindric algebra $\A\in \Nr_3\CA_{\omega}$, hence $\A$ is representable, together
with a finite atom structure of another simple cylindric algebra, that is also representable.

The former algebra will be a set algebra based on a homogeneous model, that admits elimination of quantifiers
(hence will be a full neat reduct).

Such a model is constructed using  Fraisse's methods of building models by amalgamating smaller parts.
The Boolean reduct of $\A$ can be viewed as a finite direct product of the of disjoint Boolean relativizations of $\A$.
Each component will be still uncountable; the product will be indexed by the elements of the atom structure.
The language of Boolean algebras can now be expanded by constants also indexed by the atom structure,
so that $\A$ is first order interpretable in this expanded structure $\P$ based on the finite Boolean product.
The interpretation here is one dimensional and quantifier free.

The $\Ra$ reduct of $\A$ be as desired; it will be a full $\Ra$ reduct of a full neat reduct of an $\omega$ 
dimensional algebra, hence an $\Ra$ reduct
of an $\omega$ dimensional algebra, and it has a
complete elementary equivalent subalgebra not in
$\Ra\CA_k$. (This is the same idea for $\CA$, but in this case, and the other cases of its relatives, one spare dimension suffices.)

This {\it elementary subalgebra} is obtained from $\P$, by replacing one of the components of the product with an elementary
{\it countable} Boolean subalgebra, and then giving it the same interpretation.
First order logic will not see this cardinality twist, but a suitably chosen term
$\tau_k$ not term definable in the language of relation algebras will, witnessing that the twisted algebra is not in $\Ra\CA_k$.

For $\CA$'s and its relatives, as mentioned in the previous paragraph,
we are lucky enough to have $k$ just $n+1,$
proving the most powerful result.

\begin{definition}
Let $k\geq 4$. A $k$ witness $\tau_k$ is $m$-ary term of $\CA_k$ with rank $m\geq 2$ such 
that $\tau_k$ is not definable in the language of relation algebras (so that $k$ has to be $\geq 4$)
and for which there exists a term $\tau$ expressible in the language of relation algebras, such that
$\CA_k\models \tau_k(x_1,\ldots x_m)\leq \tau(x_1,\ldots x_m).$ (This is an implication between two first order formulas using $k$-variables).

Furthermore, whenever $\A\in {\bf Cs}_k$ (a set algebra of dimension $k$) is uncountable,
and $R_1,\ldots R_m\in A$  are such that at least one of them is uncountable,
then $\tau_k^{\A}(R_1\ldots R_m)$ is uncountable as well.
\end{definition}

The following lemma, is available in \cite{Sayed} with a sketch of proof; it is fully 
proved in \cite{MLQ}. If we require that a (representable) algebra be a neat reduct,
then quantifier elimination of the base model guarantees this, as indeed illustrated in our fully proved next lemma.

\begin{lemma} Let $V=(\At, \equiv_i, {\sf d}_{ij})_{i,j<3}$ be a finite cylindric atom structure,
such that $|\At|\geq |{}^33.|$
Let $L$ be a signature consisting of the unary relation
symbols $P_0,P_1,P_2$ and
uncountably many tenary predicate symbols.
For $u\in V$, let $\chi_u$
be the formula $\bigwedge_{u\in V}  P_{u_i}(x_i)$.
Then there exists an $L$-structure $\M$ with the following properties:
\begin{enumarab}

\item $\M$ has quantifier elimination, i.e. every $L$-formula is equivalent
in $\M$ to a boolean combination of atomic formulas.

\item The sets $P_i^{\M}$ for $i<n$ partition $M$, for any permutation $\tau$ on $3,$
$\forall x_0x_1x_2[R(x_0,x_1,x_2)\longleftrightarrow R(x_{\tau(0)},x_{\tau(1)}, x_{\tau(2)}],$

\item $\M \models \forall x_0x_1(R(x_0, x_1, x_2)\longrightarrow
\bigvee_{u\in V}\chi_u)$,
for all $R\in L$,

\item $\M\models  \exists x_0x_1x_2 (\chi_u\land R(x_0,x_1,x_2)\land \neg S(x_0,x_1,x_2))$
for all distinct tenary $R,S\in L$,
and $u\in V.$

\item For $u\in V$, $i<3,$
$\M\models \forall x_0x_1x_2
(\exists x_i\chi_u\longleftrightarrow \bigvee_{v\in V, v\equiv_iu}\chi_v),$

\item For $u\in V$ and any $L$-formula $\phi(x_0,x_1,x_2)$, if
$\M\models \exists x_0x_1x_2(\chi_u\land \phi)$ then
$\M\models
\forall x_0x_1x_2(\exists x_i\chi_u\longleftrightarrow
\exists x_i(\chi_u\land \phi))$ for all $i<3$

\end{enumarab}
\end{lemma}
\begin{proof}\cite{MLQ}
\end{proof}
\begin{lemma}\label{term}
\begin{enumarab}

\item For $\A\in \CA_3$ or $\A\in \SC_3$, there exist
a unary term $\tau_4(x)$ in the language of $\SC_4$ and a unary term $\tau(x)$ in the language of $\CA_3$
such that $\CA_4\models \tau_4(x)\leq \tau(x),$
and for $\A$ as above, and $u\in \At={}^33$,
$\tau^{\A}(\chi_{u})=\chi_{\tau^{\wp(^nn)}(u).}$

\item For $\A\in \PEA_3$ or $\A\in \PA_3$, there exist a binary
term $\tau_4(x,y)$ in the language of $\SC_4$ and another  binary term $\tau(x,y)$ in the language of $\SC_3$
such that $PEA_4\models \tau_4(x,y)\leq \tau(x,y),$
and for $\A$ as above, and $u,v\in \At={}^33$,
$\tau^{\A}(\chi_{u}, \chi_{v})=\chi_{\tau^{\wp(^nn)}(u,v)}.$


\end{enumarab}
\end{lemma}

\begin{proof}

\begin{enumarab}

\item For all reducts of polyadic algebras, these terms are given in \cite{FM}, and \cite{MLQ}.
For cylindric algebras $\tau_4(x)={}_3 s(0,1)x$ and $\tau(x)=s_1^0c_1x.s_0^1c_0x$.
For polyadic algebras, it is a little bit more complicated because the former term above is definable.
In this case we have $\tau(x,y)=c_1(c_0x.s_1^0c_1y).c_1x.c_0y$, and $\tau_4(x,y)=c_3(s_3^1c_3x.s_3^0c_3y)$.

\item  We omit the construction of such terms. But from now on, we assme that they exist.
\end{enumarab}
\end{proof}

\begin{theorem}
\begin{enumarab}
\item There exists $\A\in \Nr_3\QEA_{\omega}$
with an elementary equivalent cylindric  algebra, whose $\SC$ reduct is not in $\Nr_3\SC_4$.
Furthermore, the latter is a complete subalgebra of the former.

\item Assume that there is $k$ witness. Then there exists a 
relation algebra $\A\in \Ra\CA_{\omega}$, with an elementary equivalent relation algebra not in $\Ra\CA_k$.
Furthermore, the latter is a complete subalgebra of the former.
\end{enumarab}
\end{theorem}

\begin{proof} Let $\L$ and $\M$ as above. Let
$\A_{\omega}=\{\phi^M: \phi\in \L\}.$
Clearly $\A_{\omega}$ is a locally finite $\omega$-dimensional ylindric set algebra.

The proof for $\CA$s; and its relatives are very similar. Let us prove it for $\PEA$. Here we have to add a condition to our constructed model.
Assume that the relation symbols are indexed by an uncountable set $I$.
We assume that there is a group structure on $I$, and that $R_i\circ R_j=R_{i+j}$.
We take $\At=({}^33, \equiv_i, \equiv_{ij}, d_{ij})$, where
for $u,v\in \At$, $u\equiv_i v$ iff $u$ and $v$ agree off $i$ and $v\equiv_{ij}u$ iff $u\circ [i,j]=v$. We denote $^33$ by $V$.

By the symmetry condition we have $\A$ is a $\PEA_3$, and
$\A\cong \Nr_3\A_{\omega}$, the isomorphism is given by
$\phi^{\M}\mapsto \phi^{\M}.$
Quantifier elimination in $\M$ guarantees that this map is onto, so that $\A$ is the full  neat reduct.
For $u\in {}V$, let $\A_u$ denote the relativisation of $\A$ to $\chi_u^{\M}$
i.e $\A_u=\{x\in A: x\leq \chi_u^{\M}\}.$ Then $\A_u$ is a Boolean algebra.
Also  $\A_u$ is uncountable for every $u\in V$
because by property (iv) of the above lemma,
the sets $(\chi_u\land R(x_0,x_1,x_2)^{\M})$, for $R\in L$
are distinct elements of $\A_u$.   Define a map $f: \Bl\A\to \prod_{u\in {}V}\A_u$, by
$f(a)=\langle a\cdot \chi_u\rangle_{u\in{}V}.$
We expand the language of the Boolean algebra $\prod_{u\in V}\A_u$ by constants in
such a way that
$\A$ becomes interpretable in the expanded structure (see the next proof for a detailed description of these constants)

$\P$ denotes the
structure $\prod_{u\in {}V}\A_u$ for the signature of Boolean algebras expanded
by constant symbols $1_{u}$ for $u\in V$
and ${\sf d}_{ij}$ for $i,j\in 3$ as in \cite{Sayedneat}. The closed terms corresponding to substitutions are
given by $h_S=\{v: \exists u\in S: v\equiv_{ij}u\}$.

In more detail let $\P$ denote the 
following structure for the signature of boolean algebras expanded
by constant symbols $1_u$ for $u\in {}V$ and ${\sf d}_{ij}$ for $i,j\in \alpha$: 

\begin{enumarab}
\item The Boolean part of $\P$ is the boolean algebra $\prod_{u\in {}V}\A_u$,

\item $1_u^{\P}=f(\chi_u^{\M})=\langle 0,\cdots0,1,0,\cdots\rangle$ 
(with the $1$ in the $u^{th}$ place)
for each $u\in {}V$,

\item ${\sf d}_{ij}^{\P}=f({\sf d}_{ij}^{\A})$ for $i,j<\alpha$.
\end{enumarab}

Define a map $f: \Bl\A\to \prod_{u\in {}V}\A_u$, by
$$f(a)=\langle a\cdot \chi_u\rangle_{u\in{}V}.$$

Then there are quantifier free formulas
$\eta_i(x,y)$ and $\eta_{ij}(x,y)$ such that
$\P\models \eta_i(f(a), b)$ iff $b=f(c_i^{\A}a)$ and
$\P\models \eta_{ij}(f(a), b)$ iff $b=f(s_{[i,j]}a).$

Now, like the $\CA$ case, $\A$ is interpretable in $\P$, and indeed the interpretation is one dimensional and quantifier free.
For $v\in V$, let $\B_v$ be a complete countable elementary subalgebra of $\A_v$.
Then proceed like the $\CA$ case, except that we take a different product, since we have a different atom structure, with relations
for substitutions:
Let $u_1, u_2\in V$ and let $v=\tau_(u_1,u_2)$, as given in the above lemma.
Let $J=\{u_1,u_2, s_{[i,j]v}: i, j<3\}$.
Let  $\B=\A_{u_1}\times \A_{u_2}\times \B_{v}\times \prod_{i,j<3, i\neq j} \B_{s_{[i,j]}v}\times \prod_{u\in V\sim J} \A_u$
inheriting the same interpretation. Notice that here we made all the permuted versions of $\B_v$ countable, so that $B_v$ {\it remains} countable,
because substitutions corresponding to transpositions
are present in our signature, so if one of the permuted components is uncountable, then $\B_{v}$ would be uncountable, and we do not want that.

The contradiction follows from the fact that had  $\B$ been a neat reduct, say $\B=\Nr_3\D$
then the term $\tau_3$ as in the above lemma, using $4$ variables, evaluated in $\D$ will force the component $\B_v$ to be uncountable,
which is not the case by construction.

For the second part; for relation algebras.
The $\Ra$ reduct of $\A$ is a generalized reduct of $\A$, hence $\P$ is first order interpretable in $\Ra\A$, as well.
It follows that there are closed terms and a formula $\eta$ built out of these closed terms such that
$$\P\models \eta(f(a), b, c)\text { iff }b= f(a\circ^{\Ra\A} c),$$
where the composition is taken in $\Ra\A$.
Here $\At$ defined depends on $\tau_k$ and $\tau$, so we will not specify it any further,
we just assume that it is finite.

As before, for each $u\in \At$, choose any countable Boolean elementary
complete subalgebra of $\A_{u}$, $\B_{u}$ say.
Le $u_i: i<m$ be elements in $\At$, and let $v=\tau(u_1,\ldots u_m)$.
Let $$Q=(\prod_{u_i: i<m}\A_{u_i}\times \B_{v}\times \times \B_{\breve{b}}\times \prod_{u\in {}V\smallsetminus \{u_1,\ldots u_m, v, \breve{v}\}}
\A_u), t_j)_{u,v\in {}V,i,j<3}\equiv$$
$$(\prod_{u\in V} \A_u, 1_{u,v}, {\sf d}_{ij})_{u\in V, i,j<3}=\P.$$

Let $\B$ be the result of applying the interpretation given above to $Q$.
Then $\B\equiv \Ra\A$ as relation  algebras, furthermore $\Bl\B$ is a complete subalgebra of $\Bl\A$.
Now we use essentially the same argument. We force the $\tau(u_1,\ldots u_m)$
component together with its permuted versions (because we have converse) countable;
the resulting algebra will be a complete elementary subalgebra of the original one, but $\tau_k$
will force our twisted countable component to be uncountable, arriving at a contradiction.

In more detail, assume for contradiction that $\B=\Ra\D$ with $\D\in \CA_k$.
Then $\tau_k^{\D}(f(\chi_{u_1}),\ldots f(\chi_{u_n}))$, is uncountable in $\D$.
Because $\B$ is a full $\RA$ reduct,
this set is contained in $\B.$  For simplicity assume that $\tau^{\Cm\At}(u_1\ldots u_m)=Id.$
On the other hand, for $x_i\in B$, with $x_i\leq \chi_{u_i}$, let $\bar{x_i}=(0\ldots x_i,\ldots)$ 
with $x_i$ in the $uth$ place.
Then we have
$$\tau_k^{\D}(\bar{x_1},\ldots \bar{x_m})\leq \tau(\bar{x_1}\ldots \bar{x_m})\in \tau(f(\chi_{u_1}),\ldots f({\chi_{u_m}}))
=f(\chi_{\tau(u_1\ldots u_m)})=f(\chi_{Id}).$$
But this is a contradiction, since  $\B_{Id}=\{x\in B: x\leq \chi_{Id}\}$ is  countable and $f$ is a Boolean isomorphism.
\end{proof}

\subsection{Neat atom structures}

We note that the construction here is actually stronger than the one given for finite dimensions, since it provides atomic algebras $\A$ and $\B$, 
so that we can talk about their atom structures, and it also encompasses the finite dimensional case.

$R$ be an uncountable set and let $Cof R$ be set of all non-empty finite or cofinite subsets  $R$.
Let $\alpha$ be an ordinal. For $k$ finite, $k\geq 1$, let
$$S(\alpha,k)=\{i\in {}^\alpha(\alpha+k)^{(Id)}: \alpha+k-1\in Rgi\},$$
$$\eta(X)=\bigvee \{C_r: r\in X\},$$
$$\eta(R\sim X)=\bigwedge\{\neg C_r: r\in X\}.$$


We give a construction for cylindric algebras for all dimensions $>1$.
Let $\alpha>1$ be any ordinal. $(W_i: i\in \alpha)$ be a disjoint family of sets each of cardinality $|R|$.
Let $M$ be their disjoint union, that is
$M=\bigcup W_i$. Let $\sim$ be an equvalence relation on $M$ such that $a\sim b$ iff $a,b$ are in the same block.
Let $T=\prod W_i$. Let $s\in T$, and let $V={}^{\alpha}M^{(s)}$. 
For $s\in V$, we write $D(s)$ if $s_i\in W_i$, and we let $\C=\wp(V)$.
 
\begin{lemma}
There are $\alpha$-ary relations $C_r\subseteq {}^{\alpha}M^{(s)}$ on the base $M$ for all $r\in R$,
such that conditions (i)-(v) below hold:
\begin{enumroman}
\item $\forall s(s\in C_r\implies D(s))$

\item For all $f\in {}^{\alpha}W^{(s)}$ for all $r\in R$, for all permutations
$\pi\in ^{\alpha}\alpha^{(Id)}$, if $f\in C_r$ then $f\circ \pi\in C_r.$ 

\item For all $1\leq k<\omega$, for all 
$v\in {}^{\alpha+k-1}W^{(s)}$ one to one,  for all $x\in W$, $x\in W_m$ say, then for any
function $g:S(\alpha,k)\to Cof^+R$ 
for which $\{i\in S(\alpha,k):|\{g(i)\neq R\}|<\omega\}$, 
there is a $v_{\alpha+k-1}\in W_m\smallsetminus Rgv$ such that 
and  
$$\bigwedge \{D(v_{i_j})_{j<\alpha}\implies \eta(g(i))[\langle v_{i_j}\rangle]: 
i\in S(\alpha,k)\}.$$
\item The $C_r$'s are pairwise disjoint.
\end{enumroman}
\end{lemma}
If an atom structure has one completely representable algebra, then all algebras based on this atom structure are completey representable.
Here we show that in contrast, there is an atom structure $\At$ and $\A\in \Nr_{\alpha}\K_{\alpha+\omega}$, 
$\B\notin \Nr_{\alpha}\K_{\alpha+1}$, such that $\At\A=\At\B=\At$. Futhermor $\A$ and $\B$ are not elementary equivalent.
\begin{theorem} For every ordinal $\alpha>1$, there exists an atom structure that is not neat. 
\end{theorem}
\begin{demo}{Proof}
Let $\alpha>1$ and $\F$ is field of characteristic $0$. 
Let 
$$V=\{s\in {}^{\alpha}\F: |\{i\in \alpha: s_i\neq 0\}|<\omega\},$$
Note that $V$is a vector space over the field $\F$. 
We will show that $V$ is a weakly neat atom structure that is not strongly neat.
Indeed $V$ is a concrete atom structure $\{s\}\equiv _i \{t\}$ if 
$s(j)=t(j)$ for all $j\neq i$, and
$\{s\}\equiv_{ij} \{t\}$ if $s\circ [i,j]=t$.

Let $\C$ be the full complex algebra of this atom structure, that is
$${\C}=(\wp(V),
\cup,\cap,\sim, \emptyset , V, {\sf c}_i,{\sf d}_{ij}, {\sf s}_{ij})_{i,j\in \alpha}.$$  
Then clearly $\wp(V)\in \Nr_{\alpha}\CA_{\alpha+\omega}$.
Indeed Let $W={}^{\alpha+\omega}\F^{(0)}$. Then
$\psi: \wp(V)\to \Nr_{\alpha}\wp(W)$ defined via
$$X\mapsto \{s\in W: s\upharpoonright \alpha\in X\}$$
is an isomomorphism from $\wp(V)$ to $\Nr_{\alpha}\wp(W)$.
We shall construct an algebra $\A$ such that $\At\A\cong V$ but $\A\notin \Nr_{\alpha}\CA_{\alpha+1}$.

Let $y$ denote the following $\alpha$-ary relation:
$$y=\{s\in V: s_0+1=\sum_{i>0} s_i\}.$$
Note that the sum on the right hand side is a finite one, since only 
finitely many of the $s_i$'s involved 
are non-zero. 
\end{demo}
\begin{theorem} If $\tau_k$ exists then $\A$ and $\B$ can be chosen to be atomic
\end{theorem}

\section{Neat reducts and games}

We start by characterizaing the class $\Nr_n\CA_{\omega}$ using games, or rather the atomic algebras in $\Nr_n\CA_{\omega}$ using games.
Therefore, the devised games will be played on atom structures. 
Admittedly, games played on atom structures of neat reduct miss something, for not all neat reducts are atomic, which is not
the case for example with complete representations. But such games can go very deeply into the analysis 
distingushing between various classes that
are intimately related, and hard to distinguish. So we basically use games that  are oriented to constructing counterexamples.

Our treatment in this part 
follows very closely \cite{r}. The essential difference is that in the games devised we deal with $n$ 
dimensional networks (as opposed to $2$ dimensional networks or basic matrices)
and triangle moves are replaced by what we call cylindrifier moves in the games.
Therefore, we will be rather sketch referring to the $\RA$ analogues of our results proved by Hirsch \cite{r}.
We need some prelimenaries.

\begin{definition}\label{def:string} 
Let $n$ be an ordinal. An $s$ word is a finite string of substitutions $({\sf s}_i^j)$, 
a $c$ word is a finite string of cylindrifications $({\sf c}_k)$.
An $sc$ word is a finite string of substitutions and cylindrifications
Any $sc$ word $w$ induces a partial map $\hat{w}:n\to n$
by
\begin{itemize}

\item $\hat{\epsilon}=Id$

\item $\widehat{w_j^i}=\hat{w}\circ [i|j]$

\item $\widehat{w{\sf c}_i}= \hat{w}\upharpoonright(n\sim \{i\}$ 

\end{itemize}
\end{definition}
If $\bar a\in {}^{<n-1}n$, we write ${\sf s}_{\bar a}$, or more frequently 
${\sf s}_{a_0\ldots a_{k-1}}$, where $k=|\bar a|$,
for an an arbitary chosen $sc$ word $w$
such that $\hat{w}=\bar a.$ 
$w$  exists and does not 
depend on $w$ by \cite[definition~5.23 ~lemma 13.29]{HHbook}. 
We can, and will assume \cite[Lemma 13.29]{HHbook} 
that $w=s{\sf c}_{n-1}{\sf c}_n.$
[In the notation of \cite[definition~5.23,~lemma~13.29]{HHbook}, 
$\widehat{s_{ijk}}$ for example is the function $n\to n$ taking $0$ to $i,$
$1$ to $j$ and $2$ to $k$, and fixing all $l\in n\setminus\set{i, j,k}$.]
Let $\delta$ be a map. Then $\delta[i\to d]$ is defined as follows. $\delta[i\to d](x)=\delta(x)$
if $x\neq i$ and $\delta[i\to d](i)=d$. We write $\delta_i^j$ for $\delta[i\to \delta_j]$.

\begin{definition}
From now on let $2\leq n<\omega.$ Let $\C$ be an atomic $\CA_{n}$. 
An \emph{atomic  network} over $\C$ is a map
$$N: {}^{n}\Delta\to \At\cal C$$ 
such that the following hold for each $i,j<n$, $\delta\in {}^{n}\Delta$
and $d\in \Delta$:
\begin{itemize}
\item $N(\delta^i_j)\leq {\sf d}_{ij}$
\item $N(\delta[i\to d])\leq {\sf c}_iN(\delta)$ 
\end{itemize}
\end{definition}
Note than $N$ can be viewed as a hypergraph with set of nodes $\Delta$ and 
each hyperedge in ${}^n\Delta$ is labelled with an atom from $\C$.
We call such hyperedges atomic hyperedges.

For relation algebras an atomic network, is just a basic matrix in the sense of Maddux, 
which is a map from a set of ordered pairs to the atoms of a relation algebra.
What we have defined can be viewed as a hypernetwork or, if you like, a {\it basic tensor}.
We write $\nodes(N)$ for $\Delta.$ But it can happen 
let $N$ stand for the set of nodes 
as well as for the function and the network itself. Context will help.

Define $x\sim y$ if there exists $\bar{z}$ such that $N(x,y,\bar{z})\leq {\sf d}_{01}$.
Define an equivalence relation
$\sim$ over the set of all finite sequences over $\nodes(N)$ by $\bar
x\sim\bar y$ iff $|\bar x|=|\bar y|$ and $x_i\sim y_i$ for all
$i<|\bar x|$.(It can be checked that this indeed an equivalence relation.)

(3) A \emph{ hypernetwork} $N=(N^a, N^h)$ over $\cal C$ 
consists of a network $N^a$
together with a labelling function for hyperlabels $N^h:\;\;^{<
\omega}\!\nodes(N)\to\Lambda$ (some arbitrary set of hyperlabels $\Lambda$)
such that for $\bar x, \bar y\in\; ^{< \omega}\!\nodes(N)$ 
\begin{enumerate}
\renewcommand{\theenumi}{\Roman{enumi}}
\setcounter{enumi}3
\item\label{net:hyper} $\bar x\sim\bar y \Rightarrow N^h(\bar x)=N^h(\bar y)$. 
\end{enumerate}
If $|\bar x|=k\in N$ and $N^h(\bar x)=\lambda$ then we say that $\lambda$ is
a $k$-ary hyperlabel. $(\bar x)$ is referred to a a $k$-ary hyperedge, or simply a hyperedge.
(Note that we have atomic hyperedges and hyperedges, context will help which one we intend.) 
When there is no risk of ambiguity we may drop the superscripts $a,
h$. Th labelling function for hyperlabes, labels sequences of nodes of arbitray lengths by a set of hyperlabels.
The idea, here is that a neat reduct can be viewed as a two sorted structure, the hypernetwork has to do with  the first sort, 
and the hyperlabels adjusts the algebra in higher dimensions in which the first sort, namely, the neat reduct embeds.  

The following notation is defined for hypernetworks, but applies
equally to networks.  

(4) If $N$ is a hypernetwork and $S$ is any set then
$N\restr S$ is the $n$-dimensional hypernetwork defined by restricting
$N$ to the set of nodes $S\cap\nodes(N)$.  For hypernetworks $M, N$ if
there is a set $S$ such that $M=N\restr S$ then we write $M\subseteq
N$.  If $N_0\subseteq N_1\subseteq \ldots $ is a nested sequence of
hypernetworks then we let the \emph{limit} $N=\bigcup_{i<\omega}N_i$  be
the hypernetwork defined by
$\nodes(N)=\bigcup_{i<\omega}\nodes(N_i)$,\/ $N^a(x_0,\ldots x_{n-1})= 
N_i^a(x_0,\ldots x_{n-1})$ if
$x_0\ldots x_{\mu-1}\in\nodes(N_i)$, and $N^h(\bar x)=N_i^h(\bar x)$ if $\rng(\bar
x)\subseteq\nodes(N_i)$.  This is well-defined since the hypernetworks
are nested and since hyperedges $\bar x\in\;^{<\omega}\nodes(N)$ are
only finitely long.

For hypernetworks $M, N$ and any set $S$, we write $M\equiv^SN$
if $N\restr S=M\restr S$.  For hypernetworks $M, N$, 
and any set $S$, we write $M\equiv_SN$ 
if the symmetric difference $\Delta(\nodes(M), \nodes(N))\subseteq S$ and
$M\equiv^{(\nodes(M)\cup\nodes(N))\setminus S}N$. We write $M\equiv_kN$ for
$M\equiv_{\set k}N$.

Let $N$ be a network and let $\theta$ be any function.  The network
$N\theta$ is a complete labelled graph with nodes
$\theta^{-1}(\nodes(N))=\set{x\in\dom(\theta):\theta(x)\in\nodes(N)}$,
and labelling defined by 
$(N\theta)(i_0,\ldots i_{\mu-1}) = N(\theta(i_0), \theta(i_1), \theta(i_{\mu-1}))$,
for $i_0, \ldots i_{\mu-1}\in\theta^{-1}(\nodes(N))$.  Similarly, for a hypernetwork
$N=(N^a, N^h)$, we define $N\theta$ to be the hypernetwork
$(N^a\theta, N^h\theta)$ with hyperlabelling defined by
$N^h\theta(x_0, x_1, \ldots) = N^h(\theta(x_0), \theta(x_1), \ldots)$
for $(x_0, x_1,\ldots) \in \;^{<\omega}\!\theta^{-1}(\nodes(N))$.

Let $M, N$ be hypernetworks.  A \emph{partial isomorphism}
$\theta:M\to N$ is a partial map $\theta:\nodes(M)\to\nodes(N)$ such
that for any $
i_i\ldots i_{\mu-1}\in\dom(\theta)\subseteq\nodes(M)$ we have $M^a(i_1,\ldots i_{\mu-1})= 
N^a(\theta(i), \ldots\theta(i_{\mu-1}))$
and for any finite sequence $\bar x\in\;^{<\omega}\!\dom(\theta)$ we
have $M^h(\bar x) = 
N^h\theta(\bar x)$.  
If $M=N$ we may call $\theta$ a partial isomorphism of $N$.

Hirsch played games only on relation algebra atom structures. We will play games that apply to cylindric algebra for every finite dimension, 
so that in fact we are dealing with infinitely many cases. We are infront of two choices, either explicitly refer to the dimension in our notation
(so that in some cases we will need two 'indices' one for the dimension of the algebra, the other for the number of rounds played on the atom structure
of the algebra),  or else fix it throughout. We choose the latter alternative. To simplify notation, fix $n\geq 3$.
$n$ will only appear as the dimension. It will not appear in the notation of games played; since it will be clear from context.
This simplifies notation 
considerably, and definitely permits better readability

The next definition is crucial. 
\begin{definition}
A hyperedge $\bar{x}\in {}^{<\omega}\nodes (N)$ of length $m$ is {\it short}, if there are $y_0,\ldots y_{n-1}\in \nodes(N)$, such that 
$N(x_j, y_i, \bar{z})\leq d_{01}$, for some  $j<m$, $i<n$, for some (equivalently for all)
$\bar{z}.$ Otherwise, it is called {\it long.}
A hypernetwork is called {\it $\lambda$ neat} if $N(\bar{x})=\lambda$ for all short hyper edges.
\end{definition}

Short hyperedges have to do with the atoms of the small algebra the {\it neat $n$ reduct}, which will actually be the hypernetworks. 
If $\A=\Nr_n\B$, and $\A$ is atomic, then we want the atoms of 
the $n$ neat reduct to be no smaller than the atoms of the big algebra, 
of which they are a neat reduct. This is the role of the $\lambda$ neat hypernetworks, labelling short hyperedges.
This will enable us to prove that a given atomic $n$ dimensional atomic cylindric algebra is the {\it full} neat reduct of an 
$\omega$ dimensional one.

\begin{definition}\label{def:games} Let $2\leq n<\omega$. For any $\CA_{n}$  
atom structure $\alpha$, and $n\leq m\leq
\omega$, we define two-player games $F^m(\alpha),$ \; and 
$H(\alpha)$,
each with $\omega$ rounds, 
and for $m<\omega$ we define $H_{m}(\alpha)$ with $m$ rounds.

\begin{itemize}
\item 
Let $m\leq \omega$. This is a typical $m$ pebble game.  
In a play of $F^m(\alpha)$ the two players construct a sequence of
networks $N_0, N_1,\ldots$ where $\nodes(N_i)$ is a finite subset of
$m=\set{j:j<m}$, for each $i$.  In the initial round of this game \pa\
picks any atom $a\in\alpha$ and \pe\ must play a finite network $N_0$ with
$\nodes(N_0)\subseteq  m$, 
such that $N_0(\bar{d}) = a$ 
for some $\bar{d}\in{}^{n}\nodes(N_0)$.
In a subsequent round of a play of $F^m(\alpha)$ \pa\ can pick a
previously played network $N$ an index $\l<n$, a ``face" 
$F=\langle f_0,\ldots f_{n-2} \rangle \in{}^{n-2}\nodes(N),\; k\in
m\setminus\set{f_0,\ldots f_{n-2}}$, and an atom $b\in\alpha$ such that 
$b\leq {\sf c}_lN(f_0,\ldots f_i, x,\ldots f_{n-2}).$  
(the choice of $x$ here is arbitrary, 
as the second part of the definition of an atomic network together with the fact
that $\cyl i(\cyl i x)=\cyl ix$ ensures that the right hand side does not depend on $x$).
This move is called a \emph{cylindrifier move} and is denoted
$(N, \langle f_0, \ldots f_{\mu-2}\rangle, k, b, l)$ or simply $(N, F,k, b, l)$.
In order to make a legal response, \pe\ must play a
network $M\supseteq N$ such that 
$M(f_0,\ldots f_{i-1}, k, f_i,\ldots f_{n-2}))=b$ 
and $\nodes(M)=\nodes(N)\cup\set k$.

\pe\ wins $F^m(\alpha)$ if she responds with a legal move in each of the
$\omega$ rounds.  If she fails to make a legal response in any
round then \pa\ wins. The more pebbles we have, the easier it is for \pa\ to win.

\item
Fix some hyperlabel $\lambda_0$.  $H(\alpha)$ is  a 
game the play of which consists of a sequence of
$\lambda_0$-neat hypernetworks 
$N_0, N_1,\ldots$ where $\nodes(N_i)$
is a finite subset of $\omega$, for each $i<\omega$, so that short hyperedges are al labelled by $\lambda_0$.  
In the initial round \pa\ picks $a\in\alpha$ and \pe\ must play
a $\lambda_0$-neat hypernetwork $N_0$ with nodes contained in
$\mu$ and $N_0(\bar d)=a$ for some nodes $\bar{d}\in {}^{\mu}N_0$.  
At a later stage
\pa\ can make any cylindrifier move $(N, F,k, b, l)$ by picking a
previously played hypernetwork $N$ and $F\in {}^{n-2}\nodes(N), \;l<n,  
k\in\omega\setminus\nodes(N)$ 
and $b\leq {\sf c}_lN(f_0, f_{l-1}, x, f_{n-2})$.  
[In $H(\alpha)$ we
require that \pa\ chooses $k$ as a `new node', i.e. not in
$\nodes(N)$, whereas in $F^m$ for finite $m$ it was necessary to allow
\pa\ to `reuse old nodes'. This makes the game easier as far as he is concerned.) 
For a legal response, \pe\ must play a
$\lambda_0$-neat hypernetwork $M\equiv_k N$ where
$\nodes(M)=\nodes(N)\cup\set k$ and 
$M(f_0, f_{i-1}, k, f_{n-2})=b$.
Alternatively, \pa\ can play a \emph{transformation move} by picking a
previously played hypernetwork $N$ and a partial, finite surjection
$\theta:\omega\to\nodes(N)$, this move is denoted $(N, \theta)$.  \pe\
must respond with $N\theta$.  Finally, \pa\ can play an
\emph{amalgamation move} by picking previously played hypernetworks
$M, N$ such that $M\equiv^{\nodes(M)\cap\nodes(N)}N$ and
$\nodes(M)\cap\nodes(N)\neq \emptyset$.  
This move is denoted $(M,
N)$.  To make a legal response, \pe\ must play a $\lambda_0$-neat
hypernetwork $L$ extending $M$ and $N$, where
$\nodes(L)=\nodes(M)\cup\nodes(N)$.

Again, \pe\ wins $H(\alpha)$ if she responds legally in each of the
$\omega$ rounds, otherwise \pa\ wins. 

\item For $m< \omega$ the game $H_{m}(\alpha)$ is similar to $H(\alpha)$ but
play ends after $m$ rounds, so a play of $H_{m}(\alpha)$ could be
\[N_0, N_1, \ldots, N_m\]
If \pe\ responds legally in each of these
$m$ rounds she wins, otherwise \pa\ wins.
\end{itemize}

\end{definition}

\begin{definition}\label{def:hat}
For $m\geq 5$ and $\C\in\CA_m$, if $\A\subseteq\Nr_n(\C)$ is an
atomic cylindric algebra and $N$ is an $\A$-network then we define
$\widehat N\in\C$ by
\[\widehat N =
 \prod_{i_0,\ldots i_{n-1}\in\nodes(N)}{\sf s}_{i_0, \ldots i_{n-1}}N(i_0\ldots i_{n-1})\]
$\widehat N\in\C$ depends
implicitly on $\C$.
\end{definition}
We write $\A\subseteq_c \B$ if $\A\in S_c\{\B\}$. 
\begin{lemma}\label{lem:atoms2}
Let $n<m$ and let $\A$ be an atomic $\CA_n$, 
$\A\subseteq_c\Nr_n\C$
for some $\C\in\CA_m$.  For all $x\in\C\setminus\set0$ and all $i_0, \ldots i_{n-1} < m$ there is $a\in\At(\A)$ such that
${\sf s}_{i_0\ldots i_{n-1}}a\;.\; x\neq 0$.
\end{lemma}
\begin{proof}
We can assume, see definition  \ref{def:string}, 
that ${\sf s}_{i_0,\ldots i_{n-1}}$ consists only of substitutions, since ${\sf c}_{m}\ldots {\sf c}_{m-1}\ldots 
{\sf c}_nx=x$ 
for every $x\in \A$.We have ${\sf s}^i_j$ is a
completely additive operator (any $i, j$), hence ${\sf s}_{i_0,\ldots i_{\mu-1}}$ 
is too  (see definition~\ref{def:string}).
So $\sum\set{{\sf s}_{i_0\ldots i_{n-1}}a:a\in\At(\A)}={\sf s}_{i_0\ldots i_{n-1}}
\sum\At(\A)={\sf s}_{i_0\ldots i_{n-1}}1=1$,
for any $i_0,\ldots i_{n-1}<n$.  Let $x\in\C\setminus\set0$.  It is impossible
that ${\sf s}_{i_0\ldots i_{n-1}}\;.\;x=0$ for all $a\in\At(\A)$ because this would
imply that $1-x$ was an upper bound for $\set{{\sf s}_{i_0\ldots i_{n-1}}a:
a\in\At(\A)}$, contradicting $\sum\set{{\sf s}_{i_0\ldots i_{n-1}}a :a\in\At(\A)}=1$.
\end{proof}

We now prove two theorems relating neat embeddings
to the games we defined:

\begin{theorem}\label{thm:n}
Let $n<m$, and let $\A$ be an atomic $\CA_m$
If $\A\in{\bf S_c}\Nr_{n}\CA_m, $
then \pe\ has a \ws\ in $F^m(\At\A)$. In particular, if $\A$ is countable and completely representable, then \pe has a \ws in $F^{\omega}(\At\A)$
\end{theorem}
\begin{proof}
For the first part, if $\A\subseteq\Nr_n\C$ for some $\C\in\CA_m$ then \pe\ always
plays hypernetworks $N$ with $\nodes(N)\subseteq n$ such that
$\widehat N\neq 0$. In more detail, in the initial round , let $\forall$ play $a\in \At \cal A$.
$\exists$ play a network $N$ with $N(0, \ldots n-1)=a$. Then $\widehat N=a\neq 0$.
At a later stage suppose $\forall$ plays the cylindrifier move 
$(N, \langle f_0, \ldots f_{\mu-2}\rangle, k, b, l)$ 
by picking a
previously played hypernetwork $N$ and $f_i\in \nodes(N), \;l<\mu,  k\notin \{f_i: i<n-2\}$, 
and $b\leq {\sf c}_lN(f_0,\ldots  f_{i-1}, x, f_{n-2})$.
Let $\bar a=\langle f_0\ldots f_{l-1}, k\ldots f_{n-2}\rangle.$
Then ${\sf c}_k\widehat N\cdot {\sf s}_{\bar a}b\neq 0$.
Then there is a network  $M$ such that
$\widehat{M}.\widehat{{\sf c}_kN}\cdot {\sf s}_{\bar a}b\neq 0$. Hence 
$M(f_0,\dots  k, f_{n-2})=b.$

For the second part, we have from the first part, that $\A\in S_c\Nr_n\CA_{\omega}$, the result now follows. 
\end{proof}

\begin{theorem}\label{thm:RaC}
Let $\alpha$ be a countable 
$\CA_n$ atom structure.  If \pe\ has a \ws\ in the infinite game $H(\alpha),$ then
there is a representable cylindric algebra $\C$ of
dimension $\omega$ such that $\Nr_n\C$ is atomic 
and $\At \Nr_n\C\cong\alpha$; in other words $\alpha$ is a neat atom structure.
\end{theorem}

\begin{proof} 
We shall construct a generalized atomic weak set algebra of dimension $\omega$ such that the atom 
structure of its full $n$ neat reduct is isomorphic to
the given atom structure. 

Suppose \pe\ has a \ws\ in $H(\alpha)$. Fix some $a\in\alpha$. We can define a
nested sequence $N_0\subseteq N_1\ldots$ of neat hypernetworks
where $N_0$ is \pe's response to the initial \pa-move $a$, requiring that
\begin{enumerate}
\item If $N_r$ is in the sequence and 
and $b\leq {\sf c}_lN_r(\langle f_0, f_{n-2}\rangle\ldots , x, f_{n-2})$.  
then there is $s\geq r$ and $d\in\nodes(N_s)$ such 
that $N_s(f_0, f_{i-1}, d, f_{n-2})=b$.
\item If $N_r$ is in the sequence and $\theta$ is any partial
isomorphism of $N_r$ then there is $s\geq r$ and a
partial isomorphism $\theta^+$ of $N_s$ extending $\theta$ such that
$\rng(\theta^+)\supseteq\nodes(N_r)$.
\end{enumerate}
We can schedule these requirements
to extend. To find $k$ and $N_{r+1}\supset N_r$
such that 
$N_{r+1}(f_0, k, f_{n-2})=b$ then let $k\in \omega\setminus \nodes(N_r)$
where $k$ is the least possible,
and let $N_{r+1}$ be \pe's response using her \ws, 
to the \pa move $N_r, (f_0,\ldots f_{n-1}), k, b, l).$

For an extension of the other type, let $\tau$ be a partial isomorphism of $N_r$
and let $\theta$ be any finite surjection onto a partial isomorphism of $N_r$ such that 
$dom(\theta)\cap nodes(N_r)= dom\tau$. \pe's response to \pa's move $(N_r, \theta)$ is necessarily 
$N\theta.$ Let $N_{r+1}$ be her response, using her winning strategy, to the subsequent \pa 
move $(N_r, N_r\theta).$ 

Now let $N_a$ be the limit of this sequence.
This limit is well-defined since the hypernetworks are nested.  
We shall show that $N_a$ is the base of a weak set algebra having unit  $^{\omega}N_a^{(p)}$,
for some fixed sequence $p\in {}^{\omega}N$, for which there exists a homomorphism $h$ from $\A\to \wp(N_a)$
such that $h(a)\neq 0$.

Let $\theta$ be any finite partial isomorphism of $N_a$ and let $X$ be
any finite subset of $\nodes(N_a)$.  Since $\theta, X$ are finite, there is
$i<\omega$ such that $\nodes(N_i)\supseteq X\cup\dom(\theta)$. There
is a bijection $\theta^+\supseteq\theta$ onto $\nodes(N_i)$ and $j\geq
i$ such that $N_j\supseteq N_i, N_i\theta^+$.  Then $\theta^+$ is a
partial isomorphism of $N_j$ and $\rng(\theta^+)=\nodes(N_i)\supseteq
X$.  Hence, if $\theta$ is any finite partial isomorphism of $N_a$ and
$X$ is any finite subset of $\nodes(N_a)$ then
\begin{equation}\label{eq:theta}
\exists \mbox{ a partial isomorphism $\theta^+\supseteq \theta$ of $N_a$
 where $\rng(\theta^+)\supseteq X$}
\end{equation}
and by considering its inverse we can extend a partial isomorphism so
as to include an arbitrary finite subset of $\nodes(N_a)$ within its
domain.
Let $L$ be the signature with one $n$ -ary predicate symbol ($b$) for
each $b\in\alpha$, and one $k$-ary predicate symbol ($\lambda$) for
each $k$-ary hyperlabel $\lambda$. We are working in usual first order logic.
Here we have a sequence of variables of order type $\omega$, the $n$ predicate symbols uses only $n$ variables, and roughly
the $n$ variable formulas built up out of the first $n$ variables will determine the neat reduct, the $k$ ary predicate symbols
wil determine algebras of higher dimensions as $k$ gets larger. 
This process will be interpreted in an infinite weak set algebra with base $N_a$, whose elements are 
those  assignments satisfying such formulas. 

 
For fixed $f_a\in\;^\omega\!\nodes(N_a)$, let
$U_a=\set{f\in\;^\omega\!\nodes(N_a):\set{i<\omega:g(i)\neq
f_a(i)}\mbox{ is finite}}$.
Notice that $U_a$ is weak unit (a set of sequences agreeing cofinitely with a fixed one.)


We can make $U_a$ into the universe an $L$ relativized structure $\c N_a$; 
here relativized means that we are only taking those assignments agreeing cofinitely with $f_a$,
we are not taking the standard square model.
However, satisfiability  for $L$ formulas at assignments $f\in U_a$ is defined the usual Tarskian way, except
that we use the modal notation, with assignments on the left:
For $b\in\alpha,\;
l_0, \ldots l_{n-1}, i_0 \ldots, i_{k-1}<\omega$, \/ $k$-ary hyperlabels $\lambda$,
and all $L$-formulas $\phi, \psi$, let
\begin{eqnarray*}
\c N_a, f\models b(x_{l_0}\ldots  x_{l_{n-1}})&\iff&N_a(f(l_0),\ldots  f(l_{n-1}))=b\\
\c N_a, f\models\lambda(x_{i_0}, \ldots,x_{i_{k-1}})&\iff&  N_a(f(i_0), \ldots,f(i_{k-1}))=\lambda\\
\c N_a, f\models\neg\phi&\iff&\c N_a, f\not\models\phi\\
\c N_a, f\models (\phi\vee\psi)&\iff&\c N_a,  f\models\phi\mbox{ or }\c N_a, f\models\psi\\
\c N_a, f\models\exists x_i\phi&\iff& \c N_a, f[i/m]\models\phi, \mbox{ some }m\in\nodes(N_a)
\end{eqnarray*}
For any $L$-formula $\phi$, write $\phi^{\c N_a}$ for the set of all $n$ ary assignments satisfying it; that is
$\set{f\in\;^\omega\!\nodes(N_a): \c N_a, f\models\phi}$.  Let
$D_a = \set{\phi^{\c N_a}:\phi\mbox{ is an $L$-formula}}.$ 
Then this is the universe of the following weak set algebra 
\[\c D_a=(D_a,  \cup, \sim, {\sf D}_{ij}, {\sf C}_i)_{ i, j<\omega}\] 
then  $\c D_a\in\RCA_\omega$. (Weak set algebras are representable).

Let $\phi(x_{i_0}, x_{i_1}, \ldots, x_{i_k})$ be an arbitrary
$L$-formula using only variables belonging to $\set{x_{i_0}, \ldots,
x_{i_k}}$.  Let $f, g\in U_a$ (some $a\in \alpha$) and suppose that $\{(f(i_j), g(i_j): j\leq k\}$
is a partial isomorphism of $N_a$, then one can easily prove by induction over the
quantifier depth of $\phi$ and using (\ref{eq:theta}), that
\begin{equation}
\c N_a, f\models\phi\iff \c N_a,
g\models\phi\label{eq:bf}\end{equation} 

Let $\C=\prod_{a\in \alpha} \c D_a$.  Then  $\C\in\RCA_\omega$, and $\C$ is the desired generalized weak set algebra.
Note that unit of $\C$ is the disjoint union of the weak spaces.
We set out to prove our claim. 
We shall show that $\alpha\cong \At\Nr_n\C.$

This is exactly like the corresponding proof for relation algebras; we include it for the sake of completenes and for the readers convenience.
An element $x$ of $\C$ has the form
$(x_a:a\in\alpha)$, where $x_a\in\c D_a$.  For $b\in\alpha$ let
$\pi_b:\C\to \c D_b$ be the projection defined by
$\pi_b(x_a:a\in\alpha) = x_b$.  Conversely, let $\iota_a:\c D_a\to \c
C$ be the embedding defined by $\iota_a(y)=(x_b:b\in\alpha)$, where
$x_a=y$ and $x_b=0$ for $b\neq a$.  Evidently $\pi_b(\iota_b(y))=y$
for $y\in\c D_b$ and $\pi_b(\iota_a(y))=0$ if $a\neq b$.

Suppose $x\in\Nr_n\C\setminus\set0$.  Since $x\neq 0$, 
it must have a non-zero component  $\pi_a(x)\in\c D_a$, for some $a\in \alpha$.  
Say $\emptyset\neq\phi(x_{i_0}, \ldots, x_{i_k})^{\c
 D_a}= \pi_a(x)$ for some $L$-formula $\phi(x_{i_0},\ldots, x_{i_k})$.  We
 have $\phi(x_{i_0},\ldots, x_{i_k})^{\c D_a}\in\Nr_{\mu}\c D_a)$.  Pick
 $f\in \phi(x_{i_0},\ldots, x_{i_k})^{\c D_a}$ and let $b=N_a(f(0),
 f(1), \ldots f(n-1))\in\alpha$.  We will show that 
$b(x_0, x_1, \ldots x_{n-1})^{\c D_a}\subseteq
 \phi(x_{i_0},\ldots, x_{i_k})^{\c D_a}$.  Take any $g\in
b(x_0, x_1\ldots x_{n-1})^{\c D_a}$, 
so $N_a(g(0), g(1)\ldots g(n-1))=b$.  The map $\set{(f(0),
g(0)), (f(1), g(1))\ldots (f(n-1), g(n-1))}$ is 
a partial isomorphism of $N_a$.  By
 (\ref{eq:theta}) this extends to a finite partial isomorphism
 $\theta$ of $N_a$ whose domain includes $f(i_0), \ldots, f(i_k)$. Let
 $g'\in U_a$ be defined by
\[ g'(i) =\left\{\begin{array}{ll}\theta(i)&\mbox{if }i\in\dom(\theta)\\
g(i)&\mbox{otherwise}\end{array}\right.\] By (\ref{eq:bf}), $\c N_a,
g'\models\phi(x_{i_0}, \ldots, x_{i_k})$. Observe that
$g'(0)=\theta(0)=g(0)$ and similarly $g'(n-1)=g(n-1)$, so $g$ is identical
to $g'$ over $\mu$ and it differs from $g'$ on only a finite
set of coordinates.  Since $\phi(x_{i_0}, \ldots, x_{i_k})^{\c
\ D_a}\in\Nr_{n}(\C)$ we deduce $\c N_a, g \models \phi(x_{i_0}, \ldots,
x_{i_k})$, so $g\in\phi(x_{i_0}, \ldots, x_{i_k})^{\c D_a}$.  This
proves that $b(x_0, x_1\ldots x_{n-1})^{\c D_a}\subseteq\phi(x_{i_0},\ldots,
x_{i_k})^{\c D_a}=\pi_a(x)$, and so 
$$\iota_a(b(x_0, x_1,\ldots x_{n-1})^{\c \ D_a})\leq
\iota_a(\phi(x_{i_0},\ldots, x_{i_k})^{\c D_a})\leq x\in\c
C\setminus\set0.$$  Hence every non-zero element $x$ of $\Nr_{n}\C$ 
is above
a an atom $\iota_a(b(x_0, x_1\ldots n_1)^{\c D_a})$ (some $a, b\in
\alpha$) of $\Nr_{n}\C$.  So
$\Nr_{n}\C$ is atomic and $\alpha\cong\At\Nr_{n}\C$ --- the isomorphism
is $b \mapsto (b(x_0, x_1,\dots x_{n-1})^{\c D_a}:a\in A)$.
\end{proof}

\section{The Rainbow construction, non elementary classes}

We can use such games to show that for $n\geq 3$, there is a representable $\A\in \CA_n$ 
with atom structure $\alpha$ such that $\forall$ can win the game $F^{n+2}(\alpha)$.
However, \pe\ has a \ws\ in $H_k(\alpha)$, for any $k<\omega$.
It will follow that there a cylindric algebra $\A'$ such that $\A'\equiv\A$ and \pe\ has a \ws\ in $H(\A')$.
So let $K$ be any class such that $\Nr_n\CA_{\omega}\subseteq K\subseteq S_c\Nr_n\CA_{n+2}$.
$\A'$ must belong to $\Nr_n(\RCA_\omega)=\Nr_n\CA_{\omega}$, hence $\A'\in K$.  But $\A\not\in K$
and $\A\preceq\A'$. Thus $K$ is not elementary. 

From this it easily follows that the class of completely representable cylindric algebras
is not elementary, and that the class $\Nr_n\CA_{n+k}$ for any $k\geq 0$ is not elementary either. 
Furthermore, the constructions works for many variants of cylindric algebras
like Halmos' polyadic equality algebras and Pinter's substitution algebras.
In fact, we shall prove:
\begin{theorem}\label{r} Let $3\leq n<\omega$. Then the following hold:
\begin{enumroman}
\item Any $K$ such that $\Nr_n\CA_{\omega}\subseteq K\subseteq S_c\Nr_n\CA_{n+2}$ is not elementary.
\item The inclusions $\Nr_n\CA_{\omega}\subseteq S_c\Nr_n\CA_{\omega}\subseteq S\Nr_n\CA_{\omega}$ are all proper
\end{enumroman}
\end{theorem}
(ii) Follows from the first part of the paper. The $\A$ constructed there, is in $\Nr_n\CA_{\omega}$, and $\B\in S_c\Nr_n\CA_{\omega}$ but 
$\B\notin \Nr_n\CA_{n+1}$. For the last inclusion take a countable atomic algebra in $\RCA_n$ that is not completely representable.
Then $\A\in S\Nr_n\CA_{\omega}$, but $\A\notin S_c\Nr_n\CA_{\omega}$, because it is atomic and {\it not} completely representable.

In what follows we prove the first item.
Fix finite $n>2$. We use a rainbow construction for cylindric algebras. The main difficulty here, is that 
atoms of a  cylindric algebra cannot be coded simply as binary relations. 
This makes them hard to visualize. The way to get round this is to code the atoms as coloured graphs, 
where almost all the information is coded in colours of binary relations. This makes the part of the proof dealing with labeling edges
very similar to the relation algebra case, almost identical. However, one range of colours, namely the shades of yellow,
is reserved to  to label $n-1$ tuples in the graph (Notice that if $n=3$ then the construction
is the same as relation algebras.) This also confines the $n-1$ ary coding to only one part of the construction.

The rainbow construction for cylindric algebras is an instance of the very general method of what Hirsch and Hodkinson called
constructing atom structures, hence algebras from graphs. One fixes a graph $\Gamma$. 
Then a class $K$ of structures in the signature $\Gamma\times n$, is defined
by viewing every node of such graph as a relation symbol of arity $<n$. This condition is very fortunate, 
because it allows all results proved for cylindric algebras easily transferred
to polyadic equality algebras and diagonal free reducts of cylindric algebras. 

The atom structure will actually consist of all functions $f:n\to M$, where $M\in K$, this class will be factored out suitably, to give a set, 
and the equivalence class of $f$ will be denoted by $[f]$. One can define the polyadic operations in an absolutely straightforward manner.

Properties of the graph are reflected in the complex algebra of this atom structure, 
for example $\Cm\rho(\Gamma)$ is representable, iff $\Gamma$ has infinite chromatic 
number. This is a very general construction, and achieving such equivalences at this very abstract level is definitely an achievement.
Our construction will be more tangible. Our class of modes will be coloured graphs.

We shall construct a cylindric atom structure based on {\it finite coloured graphs}, in the sense that these will constitute the atoms of the algebra.
The rainbow algebra for relation algebras was invented by by Hirsch \cite{r}.
Here, following Hodkinson,  we modify the construction, by allowing shades of yellow colours for $n-1$ tuples.
This will complicate matters a little, because such colours create {\it cones}, which are particular coloured graphs , 
and the only part of the construction
dealing the labelling of $n-1$ tuples, will depend on whether certain 
nodes are apexes of the same cone or not. 

The relation algebra constructed by Hirsch {\it does not } have an $n$ dimensional cylindric algebras. 
So basis matrices are replaced by 
$n$-coloured graphs, meaning that they have at most $n$ nodes.  

So taking $n$ coloured graphs as our the atom structure, this codes
the relation algebra 
constructed by Robin Hirsch \cite{r}. We next show that the results proved for this relation algebra atom structure 
lifts to cylindric algebras.


So let $\Z$ denotes the set of integers.
Let $P$ be the set of partial order preserving functions $f:\Z\to \N$ with $|\dom(f)|\leq 2$.

In the following the colours, the edge colours, namely the greens, whites, yellows, black and reds, 
are like the relation algebra case, 
so that we ensure that the rainbow part of the relation algebra construction
is faithfully coded, or is actually the part of the construction when we  deal only of labeling edges.
The shades of yellow are for labeling $n-1$ hyperedges.

\begin{itemize}

\item greens: $\g_i$ ($1\leq i<n-2)$, $\g_0^i$, $i\in \Z$.  

\item whites : $\w, \w_f: f\in P$
\item yellow : $\y$
\item black :  $\bb$
\item reds:  $\r_{ij}$ $(i,j\in \N)$,

\item shades of yellow : $\y_S: S\subseteq_{\omega} \N$ or $S=\N$

\end{itemize}

The above choice of atoms is very similar to the one based on $\Gamma$, as defined in Hirsch Hodkinson, with
$\Gamma$  being an infinite red clique, 
with a notational difference concerning the indicies 
of the greens with superscript
$0$, and the whites are coded by partial functions in $P$. 
These functions will help \pe\ choose the suitable whites in her game during labelling edges.
Note that {\it any } cylindric algebra based on this atom structure will be representable, because 
the chromatic number of the red clique on which it is based
is  infinite, so that its complex algebra is representable.

We should also point out that the greens are different than the relation algebra case; we have $n-2$ new greens.
These will be used to define cones, which are particular coloured graphs,  they are labels for edges in a cone. 
Such cones will play an essential role in the labeling of $n-1$ tuples.
But first we define general coloured graphs:

Using the colours above, 
\begin{definition}
A {\it coloured graph} is an undirected irreflexive graph $\Gamma$
such that every edge of $\Gamma$ is coloured by a unique 
edge colour and some $n-1$ tuples
have a unique colour too, so it is really a hypergraph.
\end{definition}

\begin{definition} 
Let $i\in \Z$, and let $\Gamma$ be a coloured graph  consisting of $n$ nodes
$x_0,\ldots  x_{n-2}, z$. We call $\Gamma$ an $i$ - cone if $\Gamma(x_0, z)=\g^0_i$   
and for every $1\leq j\leq n-2$ $\Gamma(x_j, z)=\g_j$, 
and no other edge of $\Gamma$
is coloured green.
$(x_0,\ldots x_{n-2})$ 
is called the centre of the cone, $z$ the apex of the cone
and $i$ the tint of the cone.

We define a class $\bold J$ consisting of coloured graphs 
with the following properties. 

\begin{enumarab}

\item $\Gamma$ is a complete graph.

\item $\Gamma$ contains no triangles (called forbidden triples) 
of the following types:

\vspace{-.2in}
\begin{eqnarray}
&&\nonumber\\
(\g, \g^{'}, \g^{*}), (\g_i, \g_{i}, \w), 
&&\mbox{any }i\in n-1\;  \\
(\g^j_0, \y, \w_f)&&\mbox{unless }f\in P, i\in dom(f)\\
(\g^j_0, \g^k_0, \w_0)&&\mbox{ any } j, k\in \Z\\
\label{forb:pim}(\g^i_0, \g^j_0, \r_{kl})&&\mbox{unless } \set{(i, k), (j, l)}\mbox{ is an order-}\\
&&\mbox{ preserving partial function }\Z\to\N\nonumber\\
\label{forb: black}(\y,\y,\y), (\y,\y,\bb)\\ 
\label{forb:match}(\r_{ij}, \r_{j'k'}, \r_{i^*k^*})&&\mbox{unless }i=i^*,\; j=j'\mbox{ and }k'=k^*
\end{eqnarray}
and no other triple of atoms is forbidden. This part is the coding of binary relations in the graph.
The next two conditions have to do with labeling $n-1$ tuples, 
and imposing conditions on shades of yellow used to label the base of an $i$ cone; 
$i$ has to belong to the indexing set. 

Edges are labelled like the relation algebra case, $n-1$ 
tuples are labelled by shades of yellow, 
the interaction of the two is pinned down to colouring the cones.

\item If $a_0,\ldots   a_{n-2}\in \Gamma$ are distinct, and no edge $(a_i, a_j)$ $i<j<n$
is coloured green, then the sequence $(a_0, \ldots a_{n-2})$ 
is coloured a unique shade of yellow.
No other $(n-1)$ tuples are coloured shades of yellow.

\item If $D=\set{d_0,\ldots  d_{n-2}, \delta}\subseteq \Gamma$ and 
$\Gamma\upharpoonright D$ is an $i$ cone with apex $\delta$, inducing the order 
$d_0,\ldots  d_{n-2}$ on its base, and the tuple 
$(d_0,\ldots d_{n-2})$ is coloured by a unique shade 
$y_S$ then $i\in S.$

\end{enumarab}
\end{definition}
This is the class of structures $K$ we are dealing with, every element $M$ in is a coloured graph.
and the defining relations above can be coded in first order logic, more precisely, 
every green, white, black, red, atom corresponds to a binary relation, and every $n-1$ colour is coded as an $n-1$ relations, 
and the colured graphs are defined
as the first order structures, of a set of $L_{\omega_1,\omega}$ as presented in \cite{HHbook}.

We define a cylindric algebra of 
dimension $n$. We first specify our  atom structure which will consists of surjections 
to finite coloured graphs, or rather a factroing out of this set, identifying
two coloured graphs the obvious way.

Let $$K=\{a: a \text { is a surjective map from $n$ onto some } \Gamma\in \bold J
\text { with nodes } \Gamma\subseteq \omega\}.$$
We write $\Gamma_a$ for the element of $K$ for which 
$a:n\to \Gamma$ is a surjection.
Let $a, b\in K$ define the following equivalence relation: $a \sim b$ if and only if
\begin{itemize}
\item $a(i)=a(j)\text { and } b(i)=b(j)$

\item $\Gamma_a(a(i), a(j))=\Gamma_b(b(i), b(j))$ whenever defined

\item $\Gamma_a(a(k_0)\dots a(k_{n-2}))=\Gamma_b(b(k_0)\ldots b(k_{n-1}))$ whenever  
defined
\end{itemize}
Let $\mathfrak{C}$ be the set of equivalences classes. Then define
$$[a]\in E_{ij} \text { iff } a(i)=a(j)$$
$$[a]T_i[b] \text { iff }a\upharpoonright n\sim \{i\}=b\upharpoonright n\sim \{i\}.$$
This defines a  $\CA_n$ 
atom structure.
Let $3\leq n<\omega$. 
The idea is to show that $\C_{n}$ be the complex algebra over $\mathfrak{C}$. Using the games devised above, 
we will show that $\C_{n}$ is 
not in $S_c\Nr_n\CA_{n+2}$ 
but an elementary extension of $\A$ belongs to
$\Nr_{n}\CA_{\omega}.$

The games above were formulated for networks on atom structures of cylindric algebras. 
A network has a set of nodels, and every $n$ tuple is labelled by an atom, that is a surjection 
from $n$ to a coloured graph. It is very hard to deal with such networks, so what
we do next, is to translate our games defined above on networks to games on coloured graphs.
First a general remark; the coloured graph and the corresponding unique 
network will have the same set of nodes. 

Let $N$ be an atomic $\C_{n}$ network, that is $N$ maps $n$ tuples, to surjections form $n$ to coloured graphs.
Assume that $N: {}^n\Delta\to K.$ We want to associate a coloured complete graph.  
The nodes are the same as $N$. Informally, we start by labelling edges. 
Let $x,y$ be two distinct nodes in $\Delta$, and $\bar{z}$ be {\it any} tuple in which they occur. We know
that  $N(\bar z)$ is an atom of $\C_{n}$, namely, a surjective map from $n$ to a finite coloured graph; or rather
the class of this map.  This  defines an edge colour of 
$x,y$. Using the fact that the dimension is at least $3$, 
the edge colour depends only on $x$ and $y$
not on the other elements of 
$\bar z$ or the positions of $x$ and $y$ in $\bar z$. 
So actually in the resulting coloured graph every edge is labelled by a finite surjection from $n$
to a coloured graph.

Similarly, $N$ defines shades of yellow  for certain $(n-1)$ tuples.  In this way $N$ 
translates
into a coloured graph. 

More precisely:

\begin{definition}
Let $\Gamma\in \bold J$ be arbitrary. Define the corresponding network $N_{\Gamma}$ 
on $\C_n$,
whose nodes are those of $\Gamma$
as follows. For each $a_0,\ldots a_{n-1}\in \Gamma$, define 
$N_{\Gamma}(a_0,\ldots  a_{n-1})=[\alpha]$
where $\alpha: n\to \Gamma\upharpoonright \set{a_0,\ldots a_{n-1}}$ is given by 
$\alpha(i)=a_i$ for all $i<n$. Then, as easily checked,  $N_{\Gamma}$ is an atomic $\C_{n}$ network.
Conversely, let $N$ be any non empty atomic $\C_n$ network. 
Define a complete coloured graph $\Gamma_N$
whose nodes are the nodes of $N$ as follows:  
\begin{itemize}
\item For all distinct $x,y\in \Gamma_N$ and edge colours $\eta$, $\Gamma_N(x,y)=\eta$
if and only if  for some $\bar z\in ^nN$, $i,j<n$, and atom $[\alpha]$, we have
$N(\bar z)=[\alpha]$, $z_i=x$ $z_j=y$ and the edge $(\alpha(i), \alpha(j))$ 
is coloured $\eta$ in the graph $\alpha$.

\item For all $x_0,\ldots x_{n-2}\in {}^{n-1}\Gamma_N$ and all yellows $\y_S$, 
$\Gamma_N(x_0,\ldots x_{n-2})= \y_S$ if and only if 
for some $\bar z$ in $^nN$, $i_0, \ldots  i_{n-2}<n$
and some atom $[\alpha]$, we have
$N(\bar z)=[\alpha]$, $z_{i_j}=x_j$ for each $j<n-1$ and the $n-1$ tuple 
$\langle \alpha(i_0),\ldots \alpha(i_{n-2})\rangle$ is coloured 
$\y_S.$ Then $\Gamma_N$ is well defined and is in $\bold J$.
\end{itemize}
\end{definition}
The following is then, though tedious and long,  easy to  check:

\begin{theorem}
For any $\Gamma\in \bold J$, we have  $\Gamma_{N_{\Gamma}}=\Gamma$, 
and for any $\C_n$ network
$N$, $N_{{\Gamma}_N}=N.$
\end{theorem}
This translation makes the following equivalent formulation of the 
games $F^m(\At\C_n),$ originally defined on networks. 

\begin{definition}
The new  game builds a nested sequence $\Gamma_0\subseteq \Gamma_1\subseteq \ldots $.
of coloured graphs. 
\pa\ picks a graph $\Gamma_0\in \bold J$ with $\Gamma_0\subseteq m$ and 
$|\Gamma_0|=m$. $\exists$ makes no response
to this move. In a subsequent round, let the last graph built be $\Gamma_i$.
$\forall$ picks 
\begin{itemize}
\item a graph $\Phi\in \bold J$ with $|\Phi|=m$
\item a single node $k\in \Phi$
\item a coloured garph embedding $\theta:\Phi\sim \{k\}\to \Gamma_i$
Let $F=\phi\smallsetminus \{k\}$. Then $F$ is called a face. 
\pe\ must respond by amalgamating
$\Gamma_i$ and $\Phi$ with the embedding $\theta$. In other words she has to define a 
graph $\Gamma_{i+1}\in C$ and embeddings $\lambda:\Gamma_i\to \Gamma_{i+1}$
$\mu:\phi \to \Gamma_{i+1}$, such that $\lambda\circ \theta=\mu\upharpoonright F.$
\end{itemize} 
\end{definition}

Let us halt for a minute to take our breath, because we have so many notions involved in our construction, then we discuss possibilites.
We started by a set of colours (atoms), then defined coloured graphs, which are complete graphs excluding certian triangles
corresponding to forbidden triples in defining atom structures of relation algebas. But in addition 
to the relation algebra part, which is not enough to code
$n-1$ ary tuples,  coloured graphs have also hyperegdes which are coloured by shades of yellow, but with a restriction,
namely, we have an $i$ cone, that happens
to be a subgraph
of a coloured graph $\Gamma$ having a base a set cardinality $n-1,$ then $i\in S$ where $\y_S$ is the label of the base. 
(Note that the networks have hyperedges of length $n$ that are labelled by atoms, and coloured graphs also have hyperedges, 
which we, from now on, refer to as $n-1$ tuples, to avoid confustion.
These last are coloured by shades of yellow.

The cones themselves are special coloured graphs  whose sides are labelled by the greens, 
and their base consisting of an $n-1$ tuple is coloured by shades of
yellow, as indicated above. Then we defined a cylindric algebra atom structure consisting of certain maps, 
each such map, is a surjection from $n$ to a coloured graph. 

But there is a crucial difference here that has to be pointed out between networks and coloured graphs. 
Coloured graphs have {\it edges} that has to be labelled by colours. 
Networks have have only $n-1$ tuples that has to be labelled by surjections from $n$ to coloured graphs. 
But the equivalence established above basicaly follows from the fact that the dimension
is $>2$, so that in labelling edges for a coloured graph arising from a network, 
we just take {\it any} tuple in the network containg this edge; 
this will be well defined. If $N$ is a network, and $(x,y)$ is an edge then $\Gamma_N(x,y)$ will be $\eta$, 
if  $\eta$ is the colour of the edge $(\alpha(i), \alpha(j))$
where $[\alpha]$ is the image of $N$ at {\it any} tuple $\bar{z}$ such that $z_i=x$ and $z_j=y$. This does not depend neither on the representative
$\alpha$ nor $z$.
So for nodes $x,y$ in a coloured graph edges are labelled by colours, and not surjections from $n$ to a coloured
graph. So this makes the colouring of edges in coloured graphs identical to the relation algebra case, where every edge in a network has 
a unique label. 
Had we played with networks, then it would have been really hard, 
to extend a given network to a larger one. Because, in such a case,
we would have to label every new $n$ tuple, by a coloured graph having at most $n$ nodes, the old tuples are labelled as they were,
but if we have a new node, and hence new $n$ tuples, then we would have had to label every such new $n$ tuple by a surjection from $n$ 
to a finite coloured graph, responding to every eventiality imposed by \pa s moves.
Fortunately, Hirsch and Hoskinson simplified the game considerably by dealing with coloured graphs, with no restriction on their size, except that 
they are finite, rather than dealing with labellings that involve surjective maps into coloured graphs of a fixed size.

Now let us consider the possibilities. There may be already a point $z\in \Gamma_i$ such that
the map $(k\to z)$ is an isomorphism over $F$.
In this case \ \pe does not need to extend the 
graph $\Gamma_i$, she can simply let $\Gamma_{i+1}=\Gamma_i$
$\lambda=Id_{\Gamma_i}$, and $\mu\upharpoonright F=Id_F$, $\mu(\alpha)=z$.
Otherwise, without loss of generality, 
let $F\subseteq \Gamma_i$, $k\notin \Gamma_i$.
Let ${\Gamma_i}^*$ be the colored graph with nodes $\nodes(\Gamma_i)\cup\{k\}$,
whose edges are the combined edges of $\Gamma_i$ and $\Phi$, 
such that for any $n-1$ tuple $\bar x$ of nodes of 
${\Gamma_i}^*$, the color ${\Gamma_i}^*(\bar x)$ is
\begin{itemize}
\item $\Gamma_i(\bar x)$ if the nodes of $x$ 
all lie in $\Gamma$ and $\Gamma_i(\bar x)$ is defined 
\item $\phi(\bar x)$ if the nodes of $\bar x$ all lie in 
$\phi$ and $\phi(\bar x)$ is defined
\item undefined, otherwise.
\end{itemize}
\pe\ has to complete the labeling of $\Gamma_i^*$ by adding 
all missing edges, colouring each edge $(\beta, k)$
for $\beta\in \Gamma_i\sim\Phi$ and then choosing a shade of 
white for every $n-1$ tuple $\bar a$ 
of distinct elements of ${\Gamma_i}^*$
not wholly contained in $\Gamma_i$ nor $\Phi$, 
if non of the edges in $\bar a$ is coloured green.
She must do this on such a way that the resulting graph belongs to $\bold J$.
If she survives each round, \pe\ has won the play.
Notice that \pe\ has a \ws in the in $F^m(\At(\C_n))$ if and only if 
and only if she has a wininng strategy in the 
graph games defined above. This is tedious and rather long to verify but basically routine. 

\begin{theorem} \pa\ has a winning strategy in $F^{n+2}(\At\C_n.)$
\end{theorem}
\begin{proof}
For that we show $\forall$ can win the game $F^{n+2}(\At\C_n)$. In his zeroth move, $\forall$ plays a graph $\Gamma \in \bold J$ with
nodes $0, 1,\ldots, n-1$ and such that $\Gamma(i, j) = \w (i < j <
n-1), \Gamma(i, n-1) = \g_i ( i = 1,\ldots, n), \Gamma(0, n-1) =
\g^0_0$, and $ \Gamma(0, 1,\ldots, n-2) = \y_\omega $. This is a $0$-cone
with base $\{0,\ldots , n-2\}$. In the following moves, $\forall$
repeatedly chooses the face $(0, 1,\ldots, n-2)$ and demands a node (possibly used before)
$\alpha$ with $\Phi(i,\alpha) = \g_i (i = 1,\ldots,  n-2)$ and $\Phi(0, \alpha) = \g^\alpha_0$,
in the graph notation -- i.e., an $\alpha$-cone on the same base.
$\exists$, among other things, has to colour all the edges
connecting nodes. The idea is that by the rules of the game 
the only permissible colours would be red. Using this, $\forall$ can force a
win eventually for else we are led to a a decreasing sequence in $\N$.

In more detail,
In the initial round $\forall$ plays a graph $\Gamma$ with nodes $0,1,\ldots n-1$ such that $\Gamma(i,j)=\w$ for $i<j<n-1$ 
and $\Gamma(i,n-1)=\g_i$
$(i=1, \ldots n-2)$, $\Gamma(0,n-1)=\g_0^0$ and $\Gamma(0,1\ldots n-2)=\y_{N}$.
$\exists$ must play a graph with $\Gamma_1(0,\ldots n-1)=\g_0$.
In the following move $\forall$ chooses the face $(0,\ldots n-2)$ and demands a node $n$ 
with $\Gamma_2(i,n)=\g_i$ and $\Gamma_2(0,n)=\g_0^{-1}.$
$\exists$ must choose a label for the edge $(n,n-1)$ of $\Gamma_2$. It must be a red atom $r_{mn}$. Since $-1<0$ we have $m<n$.
In the next move $\forall$ plays the face $(0, \ldots n-2)$ and demands a node $n+1$ such that  $\Gamma_3(i,n+1)=\g_i^{-2}$.
Then $\Gamma_3(n+1,n)$ $\Gamma_3(n+1,n-1)$ both being red, the indices must match.
$\Gamma_3(n+1,n)=r_{ln}$ and $\Gamma_3(n+1, n-1)=r_{lm}$ with $l<m$.
In the next round $\forall$ plays $(0,1\ldots n-2)$ and reuses the node $n-2$ such that $\Gamma_4(0,n-2)=\g_0^{-3}$.
This time we have $\Gamma_4(n,n-1)=\r_{jl}$ for some $j<l\in N$.
Continuing in this manner leads to a decreasing sequence in $\N$. 
\end{proof}

(Notice that here \pa\  needed at least $n+2$ pebbles. 
The number of pebbles, $k>n$ say, necessary for \pa\ to win the game, 
excludes {\it complete} neat embeddability of $\A$ in
an algebra with $k$ dimensions.) 

\begin{corollary} The algebra $\A$ (definition above) 
is not in ${\bf S_c}\Nr_{n}\CA_{n+2}$. 
\end{corollary} 

\begin{corollary} The algebra $\A$ is not completely representable
\end{corollary}
\begin{proof} The term algebra is countable, is not completely representable. Hence the complex algebra is not completely representable,
and so is any algebra in between based on this atom structure
\end{proof}

But, even still, things get a little bit more complicated when we have {\it hyperlabels}. The network part is translated as above to coloured graphs.
Since the graph and the network have the same nodes, then hyperlabels are simply labels for finite sequences of nodes of the graph.
We refer to the graph and the hyperlabels together as a hypergraph.

Recall from definition~\ref{def:games} that $H_k(\alpha)$ 
is the hypernetwork game with $k$ rounds. So here we have hypernetwork.
A hypernetwork consists of a network together with hyperlabels, functions from finite sequences
of nodels to a set of labels, that is every hyperedge has a label.
The translation of the games $H$ and $H_k$ to hypergraphs is as follows.
\begin{itemize}
\item Fix some hyperlabel $\lambda_0$.  $H_{k}(\alpha)$ is  a 
game the play of which consists of a sequence of
$\lambda_0$-neat hypernetworks 
$N_0, N_1,\ldots$ where $\nodes(N_i)$
is a finite subset of $\omega$, for each $i<\omega$.  

Now le us translate the game $H(\alpha)$ to coloured graphs, which we now call hypergraphs. 
The first kind of moves is very similar to the game $F^m$ without the restriction of finitely many pebbles
or nodes. The second and third moves, are easily  translatable to  coloured graphs.

$N^h$ are the hyperlabels; these we did not have in the game $F^m$, however, we will usualy deal with those
separately, and it wil turn out that it is easier to work with them, in reponse to \pa\ s moves.

Here as before a {\it neat } hypergraph means that it is constant on short hyperedges. 
A short hyperedge $\bar{x}$ consisting of nodes of the graph, is one such that 
there exists nodes $y_0,\ldots y_{n-1}$ such 
$\Gamma(x_i,y_j)\leq d_{01}$, for some $i,j$. 
Due to the correspondence established before between coloured graphs and networks, this is equivalent to the definition 
above given for networks.

We will play a $k$ rounded game on neat hypernetworks. Fix a hyperlabel $\lambda_0$ and a finite $k\geq 3$. 
A {\it neat} hypernetwork, now, is a pair $(\Gamma, N^h)$ with $\Gamma$ a coloured 
graph 
and $N^h$ is a set of functions from a finite sequences of nodes to a fixed set of labels.

Now the first move (baring in mind that we do not have only finitely many pebbles, and accordingly \pe\ cannot reuse nodes),
we get, and indeed this is reflected in the next game on coloured graphs.

\begin{definition}
$\forall$ picks a graph $\Gamma_0\in \bold J$ with $\Gamma_0\subseteq_{\omega} \omega$ and 
here we do not require that $|\Gamma_0|=n$. 
$\exists$ make no response
to this move. In a subsequent round, let the last graph built be $\Gamma_i$.
$\forall$ picks 
\begin{itemize}
\item a graph $\Phi\in \bold J$ with $|\phi|=n,$
\item a single node $k\in \Phi,$
\item a coloured graph embedding $\theta:\Phi\sim\{k\}\to \Gamma_i.$
Let $F=\phi\sim \{k\}$. Then, as before,  $F$ is called a face. 
\pe must respond by amalgamating
$\Gamma_i$ and $\phi$ with the embedding $\theta$ as before. 
In other words she has to define a 
graph $\Gamma_{i+1}\in C$ and embeddings $\lambda:\Gamma_i\to \Gamma_{i+1}$
$\mu:\phi \to \Gamma_{i+1}$, such that $\lambda\circ \theta=\mu\upharpoonright F.$

\end{itemize} 
Now we may write $N_{\Gamma}$ or simply $N$ instead of $\Gamma$, but in all cases we are dealing with {\it coloured graphs} 
that is {\it the translation} of networks. That is when we write $N$ then, $N$ will be viewed as a coloured graph.

The other moves are exactly like the relation algebra case, since the nodes of a network on $\C_n$ is the same
as that of that of the corresponding coloured graphs.

Alternatively, \pa\ can play a \emph{transformation move} by picking a
previously played coloured hypergraphs $N$ and a partial, finite surjection
$\theta:\omega\to\nodes(N)$, this move is denoted $(N, \theta)$.  \pe\
must respond with $N\theta$.  Finally, \pa\ can play an
\emph{amalgamation move} by picking previously played hypergraphs
$M, N$ such that $M\equiv^{\nodes(M)\cap\nodes(N)}N$ and
$\nodes(M)\cap\nodes(N)\neq \emptyset$  
This move is denoted $(M,
N)$.  To make a legal response, \pe\ must play a $\lambda_0$-neat
hypergraph $L$ extending $M$ and $N$, where
$\nodes(L)=\nodes(M)\cup\nodes(N)$.
Again, \pe\ wins $H(\alpha)$ if she responds legally in each of the
$\omega$ rounds, otherwise \pa\ wins. 
\end{definition}

We can  alter the rules of the game
$H_k$ slightly, to make life easier.
We impose certain restrictions on \pa (that are only apparent).

  \begin{itemize}

\item
\pa\ is only allowed to play transformation moves $(N, \theta)$ if $\theta$ is {\it injective.}

\item
\pa\ is only allowed to play an amalgamation move $(M, N)$ if for all
$m\in\nodes(M)\setminus\nodes(N)$ and all
$n\in\nodes(N)\setminus\nodes(M)$ the map $\set{(m, n)}\cup\set{(x,
x):x\in\nodes(M)\cap\nodes(N)}$ is not a {\it partial isomorphism}.  I.e. he
can only play $(M, N)$ if the amalgamated part is `as large as
possible'.
\end{itemize}  
If, as a result of these restrictions, \pa\ cannot move at
some stage then he loses and the game halts.

It is easy to check that \pa\ has a \ws\ in $H(\alpha)$ iff he has a
\ws\ with these restrictions to his moves.  Also, if \pa\ plays with
these restrictions to his moves, if \pe\ has a \ws\ then she has a
\ws\ which only directs her to play strict hypernetworks.  The same
holds when we consider $H_n(\alpha)$.  
We will assume that \pa\ plays 
according to these restrictions.
\end{itemize}



We make two comments, before working out the details, to give a general idea of the esence of the construction, and why it actually works? 

First \pe\ {\it cannot}  win the game $H$ with $\omega$ many rounds, 
because in this case \pa\ has an `infinite space'
to use essentially the winning strategy he used before, namely, 
forcing a strictly decreasing sequences of indices of reds. 

In fact, by some reflection, one can see that as far as \pe\ is concerned,  
winning $H$ is harder than winning $F^{\omega}$, in fact strictly so, which, in turn, strictly harder than winning 
$F^{n+2}$. (As we saw before, \pa\ can win the latter two games, in fact \pa\ can win any game $F^m$ with $n+2\leq m\leq \omega$.)

But if the game is truncated to any
finite $k$, we will show that she can win the game $H_k$. 
(Notice that if $F^m$ was also restricted to finitely many rounds then \pe\ would have won, by playing as he played
before, remembering that she lost in this last game because there were infinitely many rounds.)

{\it But why can \pe\  win the finite games. Basically, because the only {\it winning} 
strategy for \pa\ is to win on a red clique (like he played before), an a necessary condition for this 
is the exstence of $\omega$ round. 
In the finite case, that is when we have only finitely many rounds, as we proceed to show, \pe\  has enough colours, to respond to \pa\ moves. 
She uses white, then black, then red. We shall see that these three colours suffice, in case of labelling edges, 
that are {\it not} appexes of the same cone. To label edges that are appexes of the same cone \pe\ 
can only use reds, and this will not lead to a red clique (like in the game $F^m$, $m\geq 2$)
because the game is finite.}

However, to win the game \pe\ has to repond to {\it every} possible move of \pa\
A winning strategy here is complicated because \pa\
has  three kinds of moves, which makes it harder for \pe\ to win; she has to respond to every such move.
Besides corresponding to such moves \pe\ has to 
label also the hyperedges, in the hypergraph produced by \pa\. Furthermore, in the second move \pe\ really has no choice.

Neat hyperedges are easy. In a play \pe\ is required to play $\lambda_0$ neat hypernetworks, so she has no choice about the 
hyperlabels  for short hyperedges, these are labelled by $\lambda_0$. In response to a cylindrifier move 
all long hyperedges not incident with $k$ necessarily keep the hyperlabel they had in $N$. 

All long hyperedges incident with $k$ in $N$
are given unique hyperlabels not occuring as the hyperlabel of any other hyperedge in $N$. 

We can assume, without loss of generality, that we have infinite supply of hyperlabels of all finite arities so this is possible.
In reponse to an amalgamation move $(M,N)$ all long hyperedges whose range is contained in $\nodes(M)$ 
have hyperlabel determined by $M$, and those whose range is contained in nodes $N$ have hyperlabel determined
by $N$. If $\bar{x}$ is a long hyperedge of \pe\ s response $L$ where 
$\rng(\bar{x})\nsubseteq \nodes(M)$, $\nodes(N)$ then $\bar{x}$ 
is given  a new hyperabel, not used in any previously 
played hypernetwork and not used within $L$ as the label of any hyperedge other than $\bar{x}$.
This completes her strategy for labelling all hyperedges.

Now we turn to the real core of the construction. We shall deal now only with the graph part of the hypergraph. 
We need to label edges, and also label the $n-1$ tuples suitably by yellow shades.
Now we give \pe\ s strategy for edge labelling. This is very similar to the relation algebra case except that we deal with cylindrfier moves, 
instead of triangle ones. We need some notation and terminology  taken from \cite{r}. 

Every irreflexive edge of any hypergraph has an owner \pa\ or \pe\ namely the one who played this edge.
We call such edges \pa\ edges or \pe\ edges. Each long hyperedge $\bar{x}$ in a hypergraph $N$ 
occuring in the play has an envelope $v_N(\bar{x})$ to be defined shortly. 
In the initial round of \pa\ plays $a\in \alpha$ and \pe\ plays $N_0$ 
then all irreflexive edges of $N_0$ belongs to \pa\. There are no long hyperdeges in $N_0$. If in a later move, 
\pa\ plays the transformation move $(N,\theta)$ 
and \pe\ responds with $N\theta$ then owners and envelopes are inherited in the obvious way. 
If \pa\ plays a cylindrifier move and \pe responds with $M$ then the owner
in $M$ of an edge not incident with $k$ is the same as it was in $N$
and the envelope in $M$ of a long hyperedge not incident with $k$ is the same as that it was in $N$. 
The edges $(f,k), (k,f)$ belong to \pa\ in $M$ all edges $(l,k)(k,l)$
for $l\in \nodes(N)\sim F$ (where $F$ is the face played in the cylindrifier move) belong to \pe\ in $M$.
if $\bar{x}$ is any long hyperedge of $M$ with $k\in \rng(\bar{x})$, then $v_M(\bar{x})=\nodes(M)$.
If \pa\ plays the amalgamation move $(M,N)$ and \pe\ responds with $L$ 
then for $m\neq n\in nodes(L)$ the owner in $L$ of a edge $(m,n)$ is \pa\ if it belongs to
\pa\ in either $M$ or $N$, in all other cases it belongs to \pe\ in $L$. If $\bar{x}$ is a long hyperedge of $L$
then $v_L(\bar{x})=v_M(\bar{x})$ if $\rng(\bar{x})\subseteq \nodes(M)$, $v_L(\bar{x})=v_N(\bar{x})$ and  $v_L(\bar{x})=\nodes(M),$ otherwise.
This completes the definition of owners and envelopes.
By induction on the number of rounds one can show

\begin{athm}{Claim} Let $M, N$ ocur in a play of $H_k(\alpha)$ in which \pe\ uses default labelling
for hyperedges. Let $\bar{x}$ be a long hyperedge of $M$ and let $\bar{y}$ be a long hyperedge of $N$.
\begin{enumarab}
\item For any hyperedge $\bar{x}'$ with $\rng(\bar{x}')\subseteq _M(\bar{x})$, if $M(\bar{x}')=M(\bar{x})$
then $\bar{x}'=\bar{x}$.
\item if $\bar{x}$ is a long hyperedge of $M$ and $\bar{y}$ is a long hyperedge of $N$, and $M(\bar{x})=N(\bar{y})$ 
then there is a local isomorphism $\theta: v_M(\bar{x})\to v_N(\bar{y})$ such that
$\theta(x_i)=y_i$ for all $i<|x|$.
\item For any $x\in \nodes (M)\sim v_M(\bar{x})$ and $S\subseteq v_M(\bar{x})$, if $(x,s)$ belong to $\forall$ in $M$
for all $s\in S$, then $|S|\leq 2$.
\end{enumarab}
\end{athm}  

Now we define \pe\ s strategy for choosing the labels for edges and yellow colours for $n-1$ hyperedges.
We proceed inductively. This part is taken from \cite{r}.
Let $N_0, N_1,\ldots N_r$ be the start of a play of $H_k(\alpha)$ just before round $r+1$. 
\pe\ computes partial functions $\rho_s:Z\to N$, for $s\leq r$. 
These partial functions wil help \pe\ specify the suffix of the red atoms she has to choose in case whites and blacks do
not work, in reponse to \pa\ s move. It has to do only with labelling edges.  Inductively
for $s\leq r$ suppose

I. If $N_s(x,y)$ is green or yellow then $(x,y)$ belongs to $\forall$ in $N_s$.

II. $\rho_0\subseteq \ldots \rho_r$

III. $\dom(\rho_s)=\{i\in Z: \exists t\leq s, x,y\in \nodes(N_t), N_t(x,y)=\g_0^i\}$

IV. $\rho_s$ is order preserving: if $i<j$ then $\rho_s(i)<\rho_s(j)$. The range
of $\rho_s$ is widely spaced: if $i<j\in \dom\rho_s$ then $\rho_s(i)-\rho_s(j)\geq  3^{n-r}$, where $n-r$ 
is the number of rounds remaining in the game.

V. For $u,v,x,y\in \nodes(N_s)$, if $N_s(u,v)=\r_{\mu,\delta}$, $N_s(x,u)=\g_0^i$, $N_s(x,v)=\g_0^j$
$N_s(y,u)=N_s(y,v)=\y$ then

(a) if $N_s(x,y)\neq \w_f$ then $\rho_s(i)=\mu$ and $\rho_s(j)=\delta$

(b) If $N_s(x,y)=\w_f$ for some $f\in P$ , the $\mu=f(i)$ , $\delta=f(j)$.

VI. $N_s$ is a strict $\lambda_0$ neat hypernetwork.

To start with if $\forall$ plays $a$ in the intial round then $\nodes(N_0)=\{0,1,\ldots n-1\}$, the 
hyperedge labelling is defined by $N_0(0,1,\ldots n)=a$.

In reponse to a cylindrifier  move by \pa\, for some $s\leq r$ and some $p\in Z$, 
\pe\ must extend $\rho_r$ to $\rho_{r+1}$ so that $p\in \dom(\rho_{r+1})$
and the gap between elements of its range is at least $3^{n-r-1}$. Inductively, $\rho_r$ is order preserving and the gap betwen its elements is 
at least $3^{n-r}$, so this can be maintained in a further round.  

If \pa\ chooses non green atoms, green atoms with the same suffix, or green atom whose suffices 
already belong to $\rho_r$, there would be fewer 
elements to add to the domain of $\rho_{r+1}$, which makes it easy for \pe\ to define $\rho_{r+1}$.
Tis establishes properties $II-IV$ for round $r+1$.

Let us assume that  \pa\ played the cylindrifier  move. 
\pe\ has to choose labels for $\{(x,k), (k,x)\}$ 
$x\in \nodes(N_s)\sim F$, where $F$ is the face, and also for $n-1$ tuples so that the outcome is an $n$ coloured graph, the latter case will be dealt with separately.
Let us start with edges. She chooses labels for the edges
$(x,k)$ one at a time and then determines the reverse edges $(k,x)$ uniquely.
Property $I$ is clear since in all cases the only atoms \pe\ chooses white, black or red.
She never chooses green.

{\it  We distinguish between two case.

\begin{enumarab}
\item  if $x$ and $k$ are both apexes of cones on $F$, then \pe\ has no choice but to pick a 
red atom, that is not used before, and because the game is finite, she has enough reds; this cannot lead to an infinite clique. 
The colour she chooses is uniquely defined (as in the game $F^{n+2}$).

\item Other wise, this is not the case, so for some $i<n-1$ there is no $f\in F$ such 
that $N_s(k, f), N_s(f,x)$ are both coloured $\g_i$ or if $i=0$, they are coloured
$\g_0^l$ and $\g_0^{l'}$ for some $l$ and $l'$.
\end{enumarab}}

In the second case \pe\ uses the normal strategy in rainbow constructions. She chooses white if possible, else black and if both are not possible
she chooses red. In the last choice, which is the most tricky, 
 she uses $\rho_s$ and the suffix $f$ in $w_f$ to help her choose the suffices of red atoms.

Now we distinguish between several subcases of the second case. 
We assumed that \pa\ played the cylindrifier move $(N_s, F)$, here $F$ is the face,  in round $r+1$, that \pe\ survived till the $r$ th round,
and $x$ and $k$ are not appexes of the same cone. 

This is similar to the Hirsch's labelling edges, for networks \cite{r}.  

Let $i,j\in F$.
\begin{enumarab}

\item Suppose that it is not the case that  $N_s(x,i)$ and $N_s(x,j)$  are both green. 
Let $S=\{p\in \Z: (N_s(x,i)=\g_0^p\lor N_s(x,i)=\y\lor  N_s(x,j)=\g_0^p)\lor N_s(x,j)=\y\}.$
Then $|S|\leq 2$. \pe\ lets $N_{s+1}(x,k)=\w_f$ for some $f$ with $\dom(f)=S$.

Suppose that  $N_s(i,j)=\r_{\beta, \mu}$, $N_s(x,i)=\g_0^p$, $N_s(x,j)=\g_0^q$ for some $p,q\in Z$. 
By property $(IV)$ $f=\{(p,\beta), (q,\mu)\}$ is order preserving.
\pe\ lets $N_{s+1}(x,k)=\w_f$ in this case.

In all other cases: $N_s(i,j)$ is not red, or if it is 
then it is not the case that $N_s(x,i)$ $N_s(x,j)$ are both green, and it is not the case that $N_s(x,i)=N_s(x,j)=\y$, 
she lets $f:S\to \N$ an arbitrary order preserving function.
The only forbidden triangles involving $\w_f$ are avoided. Since \pe\ does not change green or yellow atoms to label new edges
and $N_{r+1}(x,k)=\w_f$, all triangles involving the new edge $(x,k)$ are consistent in $N_{r+1}$.
Clearly propery $VI$ holds after round $r+1$.

\item Else it is not the case that $N_s(x,i)=N_s(x,j)=\y$,  \pe\ lets $N_r(x,k)=\bb$. 
Property $V$ is not applicable in this case.
The only forbidden triple involving the atom $\bb$ is avoided, so all triangles 
$(x,y,k)$ are consistent in $N_{r+1}$ and property
$VI$ holds after round $r+1.$

\item If neither case above applies, either $N_s(x,i)=\g_p^0$ ad $N_s(x,j)=\y$
or $N_s(x,i)=\y$ and $N_s(x,j)=\g_0^p$
Assume the first case.  There are two subcases.
\begin{enumroman}
\item $N_s(i,j)\neq w_f$ for all $f\in P$. \pe\ lets $\mu=\rho_{r+1}(p)$, $\delta=\rho_{r+1}(q)$,
maintaining property $Va$. The only forbidden triples of atoms involving $r_{\mu,\delta}$ are avoided. The triple
of atoms form a traingle $(x,y,k)$ will not be forbidden since the only green edge incident 
with $k$ is $(i,k)$ and since $\rho_{r+1}$ is order preserving. 

Concerning forbidden triples involving reds. Suppose that we have 
$N_s(x,y), N_{r+1}(y, k)$ are both red for some $y\in \nodes(N_r).$
We have $y\notin \{i,j\}$ so \pe\ chose the red label $N_{r+1}(y,k)$. By her strategy we have
$N_s(i,y)=\g_t$ and $N_s(j,y)=\y$ . By property $Va$ for $N_{r+1}$ we have 
$N_{r+1}(x,y)=r_{\rho_{r+1}(p)\rho_{r+1}(t)}$ and $N_{r+1}(y,k)=r$ 
The property $VI$ holds for $N_{r+1}$

\item $N_s(i,j)=\w_f$. By consistency of $N_s$,  we have $p\in \dom(f)$ and since \pa\ s move 
we have $q\in \dom(f)$. \pe\  lets $\mu=f(p)$ $\delta=f(q)$
maintaining property
$V$ for round $r+1$. 

Concerning forbidden triples of atoms involving reds $r_{\mu,\delta}$.
Since $f$ is order preserving and since the only green edge incident
with $k$ is $(i,k)$ in $N_{r+1}$  triangles involving the new edge $(x,k)$ cannot give a forbidden triple. 

For the other case (involving reds) let $y\in \nodes(N_s)$ and suppose
$N_{r+1}(x,y)$ and  $N_{r+1}(y,k)$ are both red. As above, by her strategy we must have $N_s(y,i)=\g_t$ for some $t$ and $N_s(y,j)=\y$.
By consistency of $N_s$ we have $t\in \dom(f)$ and the current part of her strategy she lets $N_{r+1}(y,k)=\r_{f(t),f(q)}.$ 
By property $Vb$ for $N_s$
we have $N_{r+1}(x,y)=\r_{f(p), f(t)}$. 
So the triple of atoms from the triangle $(x,y,k)$ is not forbidden. This establishes propery $(VI)$ for
$N_{r+1}.$
\end{enumroman}
\end{enumarab}
We have finished with cylindrifier moves. Now we move to amalgamation moves.
Although our hypernetworks are all strict, it is not necessarily the 
case that hyperlabels label unique hyperedges - amalgmation moves can force that 
the same hyperlabel can label more than one hyperedge. 
However, within the envelope of a hyperdege $\bar{x}$, the hyperlabel $N(\bar{x})$ is unique.

We consider an amalgamation move $(N_s,N_t)$ chosen by \pa\ in round $r+1$.
\pe\ has to choose a label for each edge $(i,j)$ where $i\in \nodes(N_s)\sim \nodes(N_t)$ and $j\in \nodes(N_t)\sim (N_s)$.
This determines the label for the reverse edge.
Also \pe\ has to choose a $\y_S$ for any $n-1$ tuple $\bar{a}$, that is not contained completely in only one of $N_t$ or $N_s$.

Let $\bar{x}$ enumerate $\nodes(N_s)\cap \nodes(N_t)$
If $\bar{x}$ is short, then there are at most $n$nodes in the intersection 
and this case is similar to the cylindrifier moves.  If not, that is if $\bar{x}$ is long in $N_s$, then by the claim 
there is a partial isomorphism $\theta: v_{N_s}(\bar{x})\to v_{N_t}(\bar{x})$ fixing
$\bar{x}$. We can assume that $v_{N_s}(\bar{x})=\nodes(N_s)\cap \nodes N_t)=\rng(\bar{x})=v_{N_t}(\bar{x})$.
It remains to label the edges $(i,j)\in N_{r+1}$ where $i\in \nodes(N_s)\sim \nodes N_t$ and $j\in \nodes(N_t)\sim \nodes(N_s)$. 
Her startegy is similar to the cylindrfier  move. If $i$ and $j$ are appexes of the same cone she choose a red. If not she 
chooses  white atom if possible, else the black atom if possible, otherwise a red atom.
She never chooses a green atom, she lets $\rho_{r+1}=\rho_r$ and properies $II$, $III$, $IV$ remain true
in round $r+1$.

\begin{enumarab}  

\item There is no $x\in \nodes(N_s)\cap \nodes(N_t)$ such that $N_s(i,x)$ and $N_t(x,j)$ are both green. 
If there are nodes $u,v\in \nodes(N_s)\cap \nodes(N_t)$ such that
$N_s(u,v)=\r_{\beta,\mu}$, $N_s(i,u)=\g_0^p$, $N_s(i,v)=\g_0^q$, $N_t(u,j)=N_t(v,j)=y$ for some $\beta, \mu\in N$, $p,q\in Z$
or the roles of $i,j$ are swapped, she lets $f=\{(p,\beta), (q,\mu)\}$ and sets $N_{r+1}(i,j)=w_f$.
Since all edges labelled by green or yellow atoms belong to \pa\, we can apply the above claim to 
show that the points $u,v$ are unique so $f$ is well defined.
This is also true if $\bar{x}$ is short, since in this case there are only
two nodes
in $\nodes(N_s)\cap \nodes(N_t)$.

If there are no such $u,v$ as described then let $S=\{p\in \Z: \exists y\in \nodes(N_s)\cap \nodes(N_t), (N_s(i,y)=\g_0^p\land N_t(y,j)=y)\lor 
(N_s(i,y)=y\land N_t(y,j)=\g_0^p)\}$. Then $|S|\leq 2$. 
Let $f$ be any order preserving function and $\exists$ let $N_{r+1}=w_f$.
Property $(VI)$ holds for $N_{r+1}$ as for traingle moves.

\item Otherwise, if there is no such $x$, then she lets $N_r(i,j)=\b$. As with cylindrfier moves all properties are maintained.

\item Otherwise, there are $x,y\in \nodes(N_s)\cap \nodes(N_t)$ such that $N_s(i,x)=g_k$, $N_s(x,j)=g_l$ 
for some $k,l\in N$ and $N_s(i,y)=N_t(y,j)=y$. By the above proven claim  $x,y$ are unique. 
She labels $(i,j)$in $N_r$ with a red atom $r_{\beta,\mu}$ where
\begin{enumroman}
\item If $N_s(x,y)\neq \w_f$ for all $f\in P$, then $\beta=\rho_{r+1}(k)$, $\mu=\rho_{r+1}(l)$.
This maintains property $Va$.
\item Otherwise $N_s(x,y)=\w_f$ for some $f\in F$ and $\beta=f(k)$ $\mu=f(l)$.

\end{enumroman}
\end{enumarab}

Now we turn to coluring of $n-1$ tuples. For each tuple $\bar{a}=a_0,\ldots a_{n-2}\in N^{n-1}$ with no edge 
$(a_i, a_j)$ coloured green, then  \pe\ colours $\bar{a}$ by $\y_S$, where
$$S=\{i\in \N: \text { there is an $i$ cone in $N$ with base $\bar{a}$}\}.$$
We need to check that such labeling works.

Let us check that $(n-1)$ tuples are labeled correctly, by yellow colours. 
Let $D$ be  set of $n$ nodes, and supose that $N\upharpoonright D$ 
is an $i$ cone with apex
$\delta$ and base $\{d_0,\ldots d_{n-2}\}$, and that the tuple $(d_0,\ldots d_{n-2})$ is labelled $\y_S$ in $N$. 
We need to show that $i\in S$. If $D\subseteq N$, then inductively the graph
$N$ constructed so far is in $\bold J$, and therefore
$i\in S$. If $D\subseteq \Phi$ then as \pa chose $\Phi$ in $\bold J$ we get also $i\in S$. If neither holds, then $D$ contains $\alpha$ 
and also some
$\beta\in N\sim \Phi$. \pe\ chose the colour $N^+(\alpha,\beta)$ and her strategy ensures her that it is green. 
Hence neither $\alpha$ or $\beta$ can be the apex of the cone $N^+\upharpoonright D$, 
so they must both lie in the base $\bar{d}$. 
This implies that 
$\bar{d}$ is not yet labelled in $N^*$, so \pe\ has applied her strategy to choose the colour $\y_S$ to label $\bar{d}$ in $N^+$. 
But this strategy will have chosen $S$ containing $i$ since $N^*\upharpoonright D$ is already a cone
in $N^*$.  Also \pe\ never chooses a green edge, so all green edges of $N^+$ lie in $N^*$. 

That leaves one (hard) case, where there are two nodes $\beta, \beta',
\in N$, \pe\ colours both $(\beta, \alpha)$ and $(\beta',
\alpha)$ red, and the old edge $(\beta, \beta')$ has already been
coloured red (earlier in the game).
If $(\beta, \beta')$ was coloured by \pe\, that is \pe\ is their owner, then there is no problem. 
So suppose, for a contradiction, that $(\beta, \beta')$ was coloured by
\pe\. Since \pe\ chose red colours for $(\alpha, \beta)$
and $(\alpha, \beta')$, it must be the case that there are cones in
$N^*$ with apexes $\alpha, \beta, \beta'$ and the same base,
$F$, each inducing the same linear ordering $\bar{f} = (f_0,\ldots,
f_{n-2})$, say, on $F$. Of course, the tints of these cones may all
be different. Clearly, no edge in $F$ is labelled green, as no cone base can contain
green edges. It follows that $\bar{f}$  must be labeled by some
yellow colour, $\y_S$, say. Since $\Phi\in \bold J$, it obeys its definition, so
the tint $i$ (say) of the cone from $\alpha$ to $\bar{f}$ lies in
$S$. Suppose that $\lambda$ was the last node of $ F \cup \{ \beta,
\beta' \}$ to be created,as the game proceeded. As $ |F \cup \{
\beta, \beta' \}| = n + 1$, we see that \pa\ must have chosen
the colour of at least one edge in this : say, $( \lambda, \mu )$.
Now all edges from $\beta$ into $F$ are green,  so \pe\ is the owner of them
as well as of  $(\beta, \beta')$.

The same holds for edges from $\beta'$ to $F$. Hence $\lambda, \mu \in F$. 
We can now see that it was \pe\ who chose the colour $\y_S$ of
$\bar{f}$. For $\y_S$ was chosen in the round when $F$'s last node,
i.e., $\lambda$ was created. It could only have been chosen by
\pa\ if he also picked the colour of every edge in $F$
involving $\lambda$. This is not so, as the edge $(\lambda, \mu)$
was coloured by \pe\, and lies in $F$.
As $i \in S$, it follows from the definition of \pa\'s strategy
that at the time when $\lambda$ was added, there was already an
$i$-cone with base $\bar{f}$, and apex $N$ say.
We claim that $ F \cup \{ \alpha \}$ and $ F \cup \{ N \}$ are
isomorphic over $F$. For this, note that the only $(n - 1)$-tuples of
either $ F \cup \{ \alpha \}$ or $ F \cup \{ N \}$ with a
yellow colour are in $F$ ( since all others involve a green edge
). But this means that \pe\ could have taken $\alpha = N$ in
the current round, and not extended the graph. This is contrary to
our original assumption, and completes the proof.

{\bf atoms for $n-1$ tuples in the amalgamation move, like above}

\section{Blow up and blur}

The idea is to blow up a finite structure, replacing each 'colour or atom' by infinitely many, using blurs
to  represent the resulting term algebra, but the blurs are not enough to blur the structure of the finite structure in the complex algebra. 
Then, the latter cannot be representable due to a {\it finite- infinite} contradiction. 
This structure can be a finite clique in a graph or a finite relation algebra or a finite 
cylindric algebra.

We discuss the possibility of obtaining stronger results concerning completions, 
for example we approach the problem as to whether classes of subneat reducts are 
closed under completions, and analogous results 
for infinite dimensions.  Partial results in this direction are obtained by Sayed Ahmed, some of which will be 
mentioned below.

The main idea is to {\it split and blur}. Split  what? You can split a clique by taking $\omega$ many disjoint copies of it, 
you can split  a finite relation algebra, 
by splitting each atom into $\omega$ many, you can split a finite cylindric algebra. Generally, the splitting has to do with blowing up a finite structure
into infinitely many.

Then blur what? On this split one adds a subset of  a set of fixed in advance blurs, usually finite,
and then define an infinite atom structure, induced by the properties of the finite structure he originally started with.
It is not this atom structure that is blurred but rather the original finite structure.
This means that the term algebra built on this new atom structure, 
that is the algebra generated by the atoms,
coincides with a carefully chosen partition of the set of atoms obtained after splittig and bluring
up to minimal  deviations, so the original finite relation algebra is blurred to the extent that is invisible on this level.

The term algebra will be representable, using all such blurs as colours,
But the original algebra structure re-appears in the completion of this term algebra, that is the complex algebra of the atom structure,
forcing it to be
non representable, due to a finite-infinite discrepancy. However, if the blurs are infinite, then, they will blur also the structure
of the small algebra in the complex algebra, and the latter will be representable, inducing a complete representation of 
the term algebra.

\subsection{Main definition and examples}

We start by giving rigouous definitions of blowing up and bluring a finite structure.
In what follows, by an atom structure, we mean an atom structure of any class of completely additive Boolean algebras.

Let $N$ be a graph, in our subsequent investigations $N$ will be finite. But there is no reason to impose restriction on our next definition, which we try
keep as general as possible. By induce, we mean `define in a natural way', and we keep natural at this level of ambiguity.

\begin{definition}
\begin{enumarab}
 
\item A splitting of $N$ is a disjoint union  $N\times I$, where $I$ is an infinite set. 

\item  A blur for $N$ is any  set  $J$. 

\item An atom structure $\alpha$ is blown up and blurred if, there exists a subset $J'$ of a set $J$ of blurs, possibly empty,
such that  $\alpha$ has underlying set $X=N\times I \times J'$; the latter atom structure is called a blur of $N$ via $J$, 
and is denoted by $\alpha(N,J).$  Furthermore, every $j\in J$, induces a non-principal ultrafilter in $\wp(X)$.

\item An atom structure $\alpha(N, J)$ reflects $N$, if $N$ is faithfully represented in $\Cm\alpha(N,J)$

\item An atom structure $\alpha(N,J)$ is  weak if $\Tm\alpha(N,J)$ is representable.

\item An atom structure $\alpha(N, J)$ is very weak $\Tm\alpha(N,J)$ is not representable.

\item An atom structure $\alpha(N,J)$ is strong if  $\Cm\alpha(N,J)$ is representable.
\end{enumarab} 

\end{definition}


We give two examples of weak atom structures.
The first construction builds two relativized 
set algebras based on a certain model that is in turn a Fraisse limit of a class of certain
labelled graphs, with the labels coming from $\G\cup \{\rho\}\times n$, where $\G$ is an arbitrary graph and $\rho$ is a new colour. 
Under certain conditions on $\G$, the first set algebra can be represented on square units, the second, its completion, cannot.

\subsection{First example}



Let $\G$ be a graph. One can  define a family of labelled graphs $\cal F$ such that every edge of each graph $\Gamma\in {\cal F}$,
is labelled  by a unique label from
$\G\cup \{\rho\}\times n$, $\rho\notin \G$, in a carefully chosen way. The colour of $(\rho, i)$ is
defined to be $i$. The \textit{colour} of $(a, i)$ for $a \in \G$  is $i$.
$\cal F$ consists of all complete labelled graphs $\Gamma$ (possibly
the empty graph) such that for all distinct $ x, y, z \in \Gamma$,
writing $ (a, i) = \Gamma (y, x)$, $ (b, j) = \Gamma (y, z)$, $ (c,l) = \Gamma (x, z)$, we have:\\
\begin{enumarab}
\item $| \{ i, j, l \} > 1 $, or
\item $ a, b, c \in \G$ and $ \{ a, b, c \} $ has at least one edge
of $\G$, or
\item exactly one of $a, b, c$ -- say, $a$ -- is $\rho$, and $bc$ is
an edge of $\G$, or
\item two or more of $a, b, c$ are $\rho$.
\end{enumarab}

One forms a labelled graph $M$ which can be viewed as model of a natural signature,
namely, the one with relation symbols $R_{(a, i)}$, for each $a \in \G \cup \{\rho\}$, $i<n$ and

Then one takes a subset $W\subseteq {}^nM$, by roughly dropping assignments that do not satify $(\rho, l)$ for every $l<n$.
Formally, $W = \{ \bar{a} \in {}^n M : M \models ( \bigwedge_{i < j < n,
l < n} \neg (\rho, l)(x_i, x_j))(\bar{a}) \}.$
Basically, we are throwing away assignments $\bar{a}$ whose edges betwen two of its elements are labelled by $\rho$, and keeping those
whose edges of its elements are not.
All this can be done with an arbirary graph.

Now for particular choices of $\G$; for example if $\G$ is a certain rainbow graph, or more simply a countable infinite collection of pairwise
union of disjoint $N$ cliques with $N\geq n(n-1)/2$, or  is the graph whose nodes are the natural numbers, and the edge relation is defined by
$iEj$ iff $0<|i-j|<N$, for same $N.$ Here, the choice of $N$ is not haphazard, but it a bound of edges of complete graphs having $n$ nodes.

The relativized set algebras based on $M$, but permitting as assignments satisfying formulas only $n$ sequences in $W$ will be an atomic
representable algebra.

This algebr, call it $\A$, has universe $\{\phi^M: \phi\in L^n\}$ where $\phi^M=\{s\in W: M\models \phi[s]\}.$ (This is not representable by its definition
because its unit is not a square.) Here $\phi^M$ denotes the permitted asignments satisfyng $\phi$ in $M$.
Its completion is the relativized set algebra $\C$ with universe the larger $\{\phi^M: \phi\in L^n_{\infty,\omega}\}$,
which turns out not representable. (All logics are taken in the above signature).
The isomorphism from  $\Cm\At\A$ to $\C$ is given by $X\mapsto \bigcup X$.

Let us formulate this construction in the context of split and blur.
Take the $n$ disjoint copies of $N\times \omega=\G$.
Let $a\in \G\times n$. Then $a\in N\times \omega\times n$. 
Then for every $(a,i)$ where $a\in N\times \omega$, and $i<n$, we have an atom $R_{a,i}^{\M}\in \A$. 
The term algebra of $\A$ is generated by those. 

Hence $\N\times \omega\times n$ is the atom structure of $\A$ which can be weakly represented using the $n$ blurs, namely the set
$\{\rho, i): i<n\}$. 
The clique  $N$ appeas on the complex algebra level, forcing a finite $N$ colouring, 
so that the complex algebra cannot be representable.

We note that if $N$ is infinite, then the complex algebra (which is the completion of the algebra constructed
as above ) will be representable and so $\A$, together the term algebra will be  completely 
representable.

\subsection{The relation algebra}

We use the graph $N\times \omega$ of countably many disjoint $N$ cliques.
We define a relation algebra atom structure $\alpha(\G)$ of the form
$(\{1'\}\cup (\G\times n), R_{1'}, \breve{R}, R_;)$.
The only identity atom is $1'$. All atoms are self converse, 
so $\breve{R}=\{(a, a): a \text { an atom }\}.$
The colour of an atom $(a,i)\in \G\times n$ is $i$. The identity $1'$ has no colour. A triple $(a,b,c)$ 
of atoms in $\alpha(\G)$ is consistent if
$R;(a,b,c)$ holds. Then the consistent triples are $(a,b,c)$ where

\begin{itemize}

\item one of $a,b,c$ is $1'$ and the other two are equal, or

\item none of $a,b,c$ is $1'$ and they do not all have the same colour, or

\item $a=(a', i), b=(b', i)$ and $c=(c', i)$ for some $i<n$ and 
$a',b',c'\in \G$, and there exists at least one graph edge
of $G$ in $\{a', b', c'\}$.

\end{itemize}
$\alpha(\G)$ can be checked to be a relation atom structure. It is exactly the same as that used by Hirsch and Hodkinson, except
that we use $n$ colours, instead of just $3$. This allows the relation algebra to have an $n$ dimensional cylindric basis
and, in fact, the atom structure of $\A$ is isomorphic (as a cylindric algebra
atom structure) to the atom structure ${\cal M}_n$ of all $n$-dimensional basic
matrices over the relation algebra atom structure $\alpha(\G)$.

Indeed, for each  $m  \in {\cal M}_n, \,\ \textrm{let} \,\ \alpha_m
= \bigwedge_{i,j<n}  \alpha_{ij}. $ Here $ \alpha_{ij}$ is $x_i =
x_j$ if $ m_{ij} = 1$' and $R(x_i, x_j)$ otherwise, where $R =
m_{ij} \in L$. Then the map $(m \mapsto
\alpha^W_m)_{m \in {\cal M}_n}$ is a well - defined isomorphism of
$n$-dimensional cylindric algebra atom structures.

It can be  shown that the complex algebras of this atom structure is not representable, because its chromatic number is finite; indeed
it is exactly $N$. (This will be demonstrated below.)

But we want more. Is it possible, thatthe constructed relation algebrais {\it not} in $S\Ra \CA_{n+2}$ which is strictly smaller that $\RRA$.
The idea that could work here, is to use {\it relativized} representations. 
Algebras in $S\Ra\CA_{n+2}$ do posses representations
that are only {\it locally} square.  So is the blurring, using $n$ colours, based on $N$, namely $(\rho, i)$ $i<n$,
enough to prohibit the complex algebra to be representable in a {\it weaker} sense, which means that we have to strengthen our conditions, involving
the superscrit $2$ in the equation with $N$ and $n$. We have $N\geq n(n-1)/2$ but we need a further combinatorial property relating the triple
$(2, N, n)$

In any event, there is a finite-infinite discrepancy here, as well, no matter what kind of representation we consider,
the base has to be infinite. A representation maps the complex algebra into the powerset of a set of ordered pairs, with base $X$,
thae latter has to be infinite.  At the same time the graph has an $N$ coloring, 
and this can be used to partition the complex algebra into $(N\times n)+1$ blocks. 

But this is not enough; the idea in the classical case,
works because  one member of the partition induced by the finite colouring will be monochromatic, 
and will satisfy $(P;P)\cdot P\neq 0$, which is a contradiction.

The last condition is not guaranteed  when we have only relativized representations, because if $h$ 
is such a representation, it is not really a faithful one, in the sense that it can happen that there are $x_0, x_1,x_2\in X$,
and $(x_0,x_1)\in h(a)$, $(x_1, x_2)\in h(b)$, $(x_0, x_2)\in h(c)$, and $a, b, c\in \Cm\G$,  but $h((a;b).c)=0$
if the node $x_1$ witnessing composition, lies outside the $n$ clique determined by $x_0, x_2$,
This cannot happen in case of  classical representation. 
Finite clique is the measure of {\it squareness}. It will be defined shortly.

But we are also {\it certain } that the complex algebra is not in $S\Ra\CA_{n+k}$ for some $k\in \omega$, by the neat embedding 
theorem for relation algebras, namey,  
$\RRA=\bigcap_{k\in \omega}S\Ra\CA_{n+k}.$ 

Now, accordingly, let us keep $k$ loose, for the time being. We want to determine the least such $k$.
Remember that we required that $N\geq n(n-1)/2$, this was necessary to show that permutations of $\omega\cap \{\rho\}$ 
induces $n$ back and forth systems of partial isomorphisms of size less than $n$ in  our limiting labelled graph $M$, showing that is 
{\it strongly}
$n$ homogeneous, when viewed as a model for the language $L$. This in turn enabled us to show
that the term algebra is representable. 

The plan is to go on with the proof and see what other combinatorial properties one should impose on the relationship between
$N$, $n$ and now $k$ to prohibit even {\it a relativized representation}. 
Obviously one should  keep the condition $N\geq n(n-1)/2$ not to tamper with the
first part of the proof.

Let $\A=\Cm\G$, and assume that $V\subseteq X\times X$ is a relativized representation. 
An arbitray relativized representation, 
that is if we take any set of ordered pairs, is useless, its not what we want.

We need {\it locally} square representations that are like representations only on finite cliques of the base.
But what does locally square mean? 
A clique $C$ of $X$ is a subset of the domain $X$, that can indeed be viewed as a complete graph, in the sense that any two points in it
we have $X\models 1(x,y)$, equivalently $(x,y)\in V$, where $V$ is the unit of the relativization.
The property of $n+k$ squareness means, then for all cliques $C$ of $X$ with $|C|<n+k$, can always be extended to another clique 
having at most one more element witnessing composition, so that composition can be preserved in the representation,
but only locally. It is easier to build such representations; from the game theoretic point of view because $\forall$ 
moves are restricted by the size of cliques, which means
that the chance that exists provide a node witnessing composition is higher.



Now lets getting starting with our plan.

Assume for contradiction that $\Cm\alpha(\G)\in S\Ra \CA_{n+k}$, and $k\geq 2$. 
Then $\Cm\alpha(\G)$ has an $n+k-2$-flat representation $X$ \cite{HHbook2}  13.46, 
which is $n+k-2$ square \cite{HHbook2} 13.10. 

In particular, there is a set $X$, $V\subseteq X\times X$ and $g: \Cm\alpha(\G)\to \wp(V)$ 
such that $h(a)$ ($a\in \Cm\alpha(\G)$) is a binary relation on $X$, and
$h$ respects the relation algebra operations. Here $V=\{(x,y)\in X\times X: (x,y)\in h(1)\}$, where $1$ is the greatest element of 
$\Cm\alpha(\G)$. We write $1(x,y)$ for $(x,y)\in h(1).$ 

For any $m<\omega$, let $C_m(X)=\{\bar{a}\in {}^mX: Range(a)\text { is an $m$ clique }\}$, then $n+k-2$ squareness
means that that if $\bar{a}\in C_{n+k-2}(X)$, $r,s\in \Cm\G$, $i,j,k<n, k\neq i,j$, and $X\models (r;s)(a_i, a_j)$ then there is $b\in C_{n+k-2}(X)$ 
with $\bar{b}$ agreeing with $\bar{a}$ except possibly at $k$
such that $X\models r(b_i, b_k)$ and $X\models s(b_k, b_j)$. 

This is the definition. But it is not hard to show that this is equivalent to
 the simpler condition that 
for all cliques $C$ of $X$ with $|C|<n+k$, all $x,y\in C$ 
and $a,b\in \Cm\alpha(\G)$, $X\models (a;b)(x,y)$ 
there exists $z\in X$ such that $C\cup \{z\}$ is a clique and $X\models a(x,z)\land b(z,y)$.

Now $\G$ has a finite colouring using $N$ colours. Indeed, the map $f:N\times \omega$ defined by $f(l,i)=l$ 
is a finite colouring using $N$ colours. For $Y\subseteq N\times \omega$ and $l<n$ define
$(Y, k)=\{(a,i,l): (a,i)\in Y\}$, regarded as a subset of $\Cm\G$. 

The nodes of $N\times \omega$ can be partitioned into
sets $\{C_j: j<n\}$ such that there are no edges within $C_j$.
Let $J=\{1', (C_j,k): j<N, k<n\}$ 
Then clearly, $\sum J=1$ in $\Cm\alpha(\G)$,  so that $J$ is partition of $\Cm\alpha(\G)$ into $N\times n+1$ blocks.

As $J$ is finite, we have for any $x,y\in X$ there is a $P\in J$ with
$(x,y)\in h(P)$. Since $\Cm\alpha(\G)$ is infinite then $X$ is infinite. 
Ramsys's theorem aplies in this context, to allow us to infer, that there are distinct
$x_i\in X$ $(i<\omega)$, $J\subseteq \omega\times \omega$ infinite
and $P\in J$ such that $(x_i, x_j)\in h(P)$ for $(i, j)\in J$, $i\neq j$. 
Then $P\neq 1'$. 

The condition we need on $k$, is that if $(x_0, x_1)\in h(a)$, $(x_1, x_2)\in h(b)$ and 
$(x_0, x_2)\in h(c),$ then $a;b.c\neq 0$.

So this prompts: 

{\it Find a combinatorial relation between $n,k, N$ with $N\geq n(n-1)/2$
that forces $(P;P)\cdot P\neq 0$. What is the least such $k$? This is formulated for any $P$, but maybe the condition
would also force Ramseys theorem to give the right 
block.} 


A non -zero element $a$ of $\Cm\alpha(\G)$ is monochromatic, if $a\leq 1'$,
or $a\leq (\Gamma,s)$ for some $s<n$. 
Now  $P$ is monochromatic, and the $C_j$ s are independent, it follows also from the definition of $\alpha$ that
$(P;P)\cdot P=0$. 

$\Cm\alpha(\G)$ is not in $S\Ra\CA_{n+m}$, and from this, it will follow that  $\Cm\M_n\notin S\Nr_n\CA_{n+m}$, for al $m\geq k$.
Showing that the latter two cases are not closed under completions.


For a relation algebra $\R$ having an $n$ dimensional cylindric basis, let ${\sf Mat}_n\R$ be the term cylindric algebra 
of dimension $n$ generated by the basic matrices.

\begin{theorem} Let $\G$ be a graph that is a disjoint union of cliques having size $n$. Then there is a strongly $n$ homogenious
labelled graph $M$, every edge is labelled by an element from $\G\cup \{\rho\}\times n$,
$W\subseteq {}^nM$, such that the set algebra based on $W$ is an atomic $\A\in \RCA_n$, and there is an atomic 
$\R\in \RRA$ the latter with an $n$ dimensional cylindric basis, 
such that $\A\cong {\sf Mat}_n\R$, and the completions of $\A$ and $\R$ are not representable, hence they are not completely
representable.
\end{theorem}

\subsection{ Second Example}


Here we turn to our second split and blur construction.
It is a simplified version of the proof of Andr\'eka and N\'emeti, except that for a set of blurs $J$, 
they defined infinitely many tenary relations on $\omega$ with suffixes from $J$, to synchronize the composition operation. 
This was necessary to show
that the required algebras are generated by a single element; here we  use one uniform relation, 
and we sacrifize with this part of the result, which is worthwhile, due to
the reduction of the complexity of the proof. We think that
our simplified version captures the essence
of  the blow  up and blur construction of Andr\'eka and N\'emeti.

Let $I$ and $J$ be sets, for the time being assume they are finite.
We will define two partitions $(H^P: P\in I)$ and $(E^W: W\in J)$ of a given infinite set $H$,
using atoms from a finite relation algebra for the first superscripts, and "blurs" (literally) for the second superscript.

The blurs do two things. They are just enough to distort the structure of $\bold M$ in the term algebra, but not in its completions,
but at the same time they are colours that are necessary 
for representing the term algebra.

Indeed, we use the first partition to show that the complex algebra of our atom structure is not representable,
while we use the second to show that the term
algebra is representable.


Let us start getting more concrete. Let $I$ be a finite set with $|I|\geq 6$. Let $J$ be the set of all $2$ element
subsets of $I$, and let
$$H=\{a_i^{P,W}: i\in \omega, P\in I, W\in J, P\in W\}.$$
In a minute we will get even more concrete by choosing a specific finite relation relation $\bold M$ with certain properties, namely,
it cannot be represented on infinite sets. The atoms of $\bold M$ will be  $I$. This algebra is finite, so it cannot do what we want. 
A completion of a finite algebra is itself.

The index $i$ here  says that we will replace each atom of this relation algebra by infinitely many atoms, that will define an atom structure
of a new infinite relation algebra, {\it the desired algebra}.
(This is an instance of a technique called {\it splitting}, which involves splitting an atom into smaller atoms.
Invented by Andreka, it is very useful in proving non representability  results).

The structure of $\bold M$ will be {\it blown up} by splitting the atoms, then 'blurred' in the term algebra, but it will {\it not} be blurred
in the  completion of the term algebra. More precisely, $\bold M$  will be a subalgebra of the completion,
but it may (and will not be) a subalgebra
of the term algebra.

The best way to visualize the partitions we will define is to imagine that the atoms of the new algebras, form a partition of
an infinite rectangle with finite base $I$ and side $\omega$ reflecting the infinite splitting of $I$.
Or to view it as an infinite tenary martrix,  with each entry indexed by $(i, P,W)\in \omega\times I\times J$, $P\in W$.

We now define two finite partitions of the rectangle, namely $H$.
For $P\in I$, let
$$H^P=\{a_i^{P,W}: i\in \omega, W\in J, P\in W\}.$$
The finite relation algebra will be embedable in the completion via $P\mapsto H^P$,
no distortion involved. $\bold M$ will still be up there on the global level.

The $J$s are the blurs, for $W\in J,$ let
$$E^W=\{a_i^{P,W}: i\in \omega, P\in W\}.$$
The singletons will generate this partition up to a `finite blurring'.
That is the term algebra will consist of all those $X$ such that $X$ intersects $E^W$ finitely or cofinitely.
For each $W\in J$, we have $W\subseteq I$, and so $E^W$ will be the subrectangle of $H$ on the base $W$.

To implement our plan we further need  a tenary relation, which synchronizes composition; it will
tell us  which rows in the rectangle, allow composition like $\bold M$.

For $i,j,k\in \omega$ $e(i,j,k)$ abbreviates that $i,j,k$ are {\it evenly distributed}, i.e.
$$e(i,j,k)\text { iff } (\exists p,q,r)\{p,q,r\}=\{i,j,k\}, r-q=q-p$$
For example $3,5,7$ are evenly distributed, but $3,5,8$ are not.
All atoms are self-converse. This always makes life easier.
We define the consistent triples as follows
(Involving identity are as usual $(a, b, Id): a\neq b).$

Let $i,j,k\in \omega$, $P,Q,R\in I$ and $S,Z,W\in J$ such that
$P\in S$, $Q\in Z$ and $R\in W$. Then the triple
$(a_i^{P,S},a_j^{Q,Z}, a_k^{R,W})$ is consistent iff
either
\begin{enumroman}
\item $S\cap Z\cap W=\emptyset,$
or
\item $e(i,j,k)\&P\leq Q;R.$
\end{enumroman}
The second says that if $i,j,k$ are $e$ related then the composition of $P$, $Q$ and $R$
existing on those three rows, is defined like $\bold M$.

Let $\cal F$ denote this atom structure, ${\cal F}=H\cup \{Id\}$

Now, as promised, we choose a (finite) relation algebra $\bold M$ with atoms $I\cup \{1d\}$ such that
for all $P,Q\in I$, $P\neq Q$ we have
$$P;P=\{Q\in I: Q\neq P\}\cup \{Id\}\text { and } P;Q=H$$
Such an $\bold M$ exists It is also known that
$\bold M$, if representable, can be only represented on finite sets. 
Now using the above partitions we show:
\begin{theorem}
\begin{enumarab}
\item $\Cm{\cal F}$ is a relation algebra that is not representable.
\item ${\cal R}$ the term algebra over $\cal F$ is representable.
\end{enumarab}
\end{theorem}
\begin{proof}
\begin{enumarab}
\item Non representabiliy uses the first partition of $H$. Note that $;$ is defined on $\Cm({\cal F})$ so  that
$$H^P; H^Q=\bigcup \{H^Z: Z\leq P; Q\in {\bold M}\}.$$
So $\bold M$ is isomorphic to a subalgebra of $\Cm F$. But $\Cm F$
can only be represented on infinite sets, while $\M$ only on finite ones, hence we are done.

\item The representability of the term algebra uses the second partition.
The blow up and blur algebra is ${\cal R}=\{X\subseteq F: X\cap E^W\in Cof(E^W), \forall W\in J\}$.
For any $a\in F$ and $W\in J$, let
$$U^a=\{X\in R: a\in X\}$$
and
$$U^W=\{X\in R: |Z\cap E^W|\geq \omega\}$$
Let
$$\Uf=\{U^a: a\in F\}\cup \{U^W: W\in J: |E^W|\geq \omega \}.$$
$\Uf$ denotes the set of ultrafilters of $\cal R$, that include at least one non-principal
ultrafilter, that is an element of the form $U^W$.

Let $F,G,K$ be boolean ultrafilters 
in a relation algebra and let $;$
denote composition. 
Then $$F;G=\{X;Y: X\in F, Y\in G\}.$$
The triple $(F,G,K)$ is {\it consistent} 
if the following holds:
$$F;G\subseteq K, F;K\subseteq G\text{  and }G;K\subseteq F.$$

So to represent $\cal R$ using $Uf$ as colours, we want to achieve (i) -(iii) below:

\begin{enumroman}

\item $(U^a,U^b,U^W)$ is consistent whenever $a,b\in H$ and $a;b\in U^W.$

\item $(F,G,K)$ is consistent whenever at least two of $F,G,K$ are non-principal
and $F,G,K\in Uf-\{U^{Id}\}.$

\item For any $a,b,c,d\in H$, there is $W\in J'$ such that $a;b\cap c;d\in U^W$.

\end{enumroman}

Let us see how to represent this algebra.
We call $(G,l)$ a  {\it consistent coloured graph} if $G$ is a set,
$l:G\times G\to Uf$
such that for all $x,y,z\in G$, the following hold:

\begin{enumroman}

\item $l(x,y)=U^{Id}$ iff $x=y,$

\item $l(x,y)=l(y,x)$

\item The triple $(l(x,y),l(x,z),l(y,z))$ is consistent.

\end{enumroman}

We say that a consistent coloured graph $(G,l)$ is complete if for all
$x,y\in G$, and $F,K\in Uf$, whenever
$(l(x,y),F,K)$ is consistent, then there is a node $z$ such that
$l(z,x)=F$ and $l(z,y)=K$.
We will build a complete consistent graph step-by-step. So assume (inductively)
that$(G,l)$ is a consistent coloured graph and $(l(x,y), F, K)$ is a consistent triple.
We shall extend $(G,l)$ with a new point $z$ such that $(l(x,y), l(z,x), l(z,y)) =(l(x,y),G,K).$
Let $z\notin G.$ We define $l(z,p)$ for $p\in G$ as follows:
$$l(z,x)=F$$
$$l(z,y)=K, \text { and if } p\in G\smallsetminus \{x,y\}, \text { then }$$
$$l(z,p)=U^W\text { for some }W\in J' \text { such that both }$$
$$(U^W, F, l(x,p))\text { and } (U^W, K, l(y,p))\text { are consistent }.$$
Such a $W$ exists by our assumptions (i)-(iii).
Conditions (i)-(ii) guarantee that this extension is again a consistent coloured graph.

We now show that any non-empty complete coloured graph $(G,l)$ gives a representation
for $\cal R.$ For any $X\in R$ define
$$rep(X)=\{(u,v)\in G\times G: X\in l(u,v)\}$$
We show that
$$rep:{\cal R}\to R(G)$$
is an embedding. $rep$ is a boolean homomorphism
because all the labels are ultrafilters.
$$rep(Id)=\{(u,u): u\in G\},$$
and for all $X\in R$,
$$rep(X)^{-1}=rep(X).$$
The latter follows from the first condition in the definition of a consistent coloured graph.
From the second condition in the definition of a consistent coloured graph, we have:
$$rep(X);rep(Y)\subseteq rep(X;Y).$$
Indeed, let $(u,v)\in rep(X), (v,w)\in rep(Y)$ I.e. $X\in l(u,v), Y\in l(v,w).$
Since $(l(u,v), l(v,w), l(u,w))$ is consistent, then $X;Y\in l(u,w)$, i.e. $(u,w)\in
rep(X;Y).$
On the other hand, since $(G,l)$ is complete and because (i)-(ii) hold, we have:
$$rep(X;Y)\subseteq rep(X);rep(Y),$$
because $(G,l)$ is complete and because (i) and (ii) hold.
Indeed, let $(u,v)\in rep(X;Y)$. Then $X;Y\in l(u,v)$.
We show that there are $F,K\in \Uf$ such that
$$X\in F, Y\in K \text { and } (l(u,v), F,K)\text { is consistent }.$$
We distinguish between two cases:

{\bf Case 1}. $l(u,v)=U^a$ for some $a\in F$. By $X;Y\in U^a$ we have $a\in X;Y.$
Then there are $b\in X$, $c\in Y$ with $a\leq b;c.$ Then $(U^a, U^b, U^c)$ is consistent.

{\bf Case 2.}
$l(u,v)=U^W$ for some $W\in J'$. Then $|X;Y\cap E^W|\geq \omega$
by $X;Y\in U^W$.
Now if both $X$ and $Y$ are finite, then there are $a\in X$, $b\in Y$ with
$|a;b\cap E^W|\geq \omega$.
Then $(U^W, U^a, U^b)$ is consistent by (i). Assume that one of $X,Y$, say $X$ is infinite.
Let $S\in J'$ such that $|X\cap E^S|\geq \omega$ and let $a\in Y$ be arbitrary.
Then $(U^W, U^S, U^a)$ is consistent by (ii)
and $X\in U^S, Y\in U^a.$

Finally, $rep$ is one to one because $rep(a)\neq \emptyset$ for all $a\in A$.
Indeed $(u,v)\in rep(Id)$ for any
$u\in G$. Let $a\in H$. Then $(U^{Id}, U^a, U^a)$ is consistent, so there is a $v\in G$ with
$l(u,v)=U^a$. Then $(u,v)\in rep(a).$

\end{enumarab}

\end{proof}

\subsection{The cylindric algebra}

We define the atom structure like we did before. The basic matrices of the atom structure above form a
$3$ dimensional cylindric algebra. We want an
$n$ dimensional one.  Our previous construction of the atom structure satisfied (*)
satisfies $(\forall a_1\ldots a_3\ b_1\ldots b_3\in I)(\exists W\in J)(a_1;b_1)\cap\ldots (a_3;b_3)\in U^W.$

We strengthen this condition to (**)
$$(\forall a_1\ldots a_nb_1\ldots b_n\in I)(\exists W\in J)W\cap (a_1;b_1)\cap\ldots  (a_n;b_n)\neq \emptyset.$$
(This is referred to in \cite{Sayed} as an {\it $n$ complex blur for $\bold M$},
our first construction was a $3$ complex blur).

This condition will entail that the set of {\it all} $n$ by $n$ matrices is a cylindric basis on the new relation
algebra ${\cal R}_n$ defined as before, with minor modifications.

Now ${\cal R}_n$ is defined by taking $I$ be a finite set with $|I|\geq 2n+2$, $J$ be the set of all $2$ (See the proof) element
subsets of $I$. And then define everything as before.
The resulting cylindric algebra is also called the blow up and blur {\it cylindric algebra of dimension $n$},
which actually blows up  and blurs the $n$ dimensional finite cylindric algebra consisting of $n$ basic matrices of $\bold M$, whch is representable,
so such  an algebra exists for every $n.$

The new condition (**)
guarantess the amalgamation property of matrices (corresponding to commutativity of
cylindrifiers)  which is the essential  property of basis.

We know that the term algebra is a subneat reduct of an algebra in $\omega$ extra dimensions.
But we need a final tick so that the the term cylindric algebra is a {\it full} neat reduct.
This requires a yet another strenghthenig of  (**) by replacing $\exists$ by $\forall$.

Now under this stronger condition, let $\B_n$ be the set of basic matrices of our blown up and blurred ${\cal R}_n$.
In the first order language $L$ of $(\omega, <)$, which has quantifier elimination,
diagrams are defined for each $K\subseteq n$ and $\phi\in L$, via maps  $\bold e:K\times K\to {\cal R}_n$.
For an atom let $v(a)$ be its ith co-ordinate, or its $i$ th level in the rectangle.

The pair $\bold e$ and $\phi$ defines an element in $\Cm\B_n$, called a diagram,
that is a set of matrices, defined by
$$M(\bold e, \phi)=\{m\in B_n, i,j\in K, m_{ij}\leq \phi(e_{ij},v(m_{ij})\}.$$

A normal diagram is one whose entries are either atoms or finitely many blurs (by (J)), that is elements of the form $E^W$,
in addition to the condition that $\phi$ implies $\phi_e$.
Any diagram can be approximated by normal ones; and atcually it is a finite union of normal diagrams.
The term algebra turns, denoted by $\Bb_n(\bold M, J, e)$, consists of those diagrams,
and finally we get  that  that for $t<n$
$$\Nr_t \Bb_n(\bold M, J, e)\cong \Bb_t(\bold M, J, e).$$
Here actually we are also blowing and bluring the finite dimensional cylindric algebra atom consisting of matrices on $\bold M$, we blow 
up every $n$ dimensional matrix to infinitely many, where each entry is either an atom of the relation algebra or a blur;
these are exactly 
the diagrams.

\subsection{The analogy, first informaly, then formally in a map}

This construction actually has a lot of affinity with the first  model theoretic construction.
First they both prove the same thing; the Andreka et all construction  proves that in addition the term algebra is a $k$ neat reduct.
Now here we are comparing a relation algebra construction with a cylindric algebra one, but the analogy is worthwhile pointing out.

Replace the clique $N$ in Hodkinson's construction by $\bold M$, in this case the term algebra, 
$\cal R$ is also obtained by replacing every atom by infinitely many ones, 
and $\bold M$ appears on the global level as a subalgebra of the complex algebra.

To this consruction we can also associate a finite graph with finite chromatic number, namely the complete graph on $\bold M$.
The blurs are the colours, that correspond to the colours $(\rho, i)$ in Sayed Ahmed 's construction.  

In the first case the splitting of the clique $N$, uses just one index, in the second we use two indices, the atoms of $\bold M$ and the blurs.
The first partition replaces the use of Ramseys theorem, the second partition, is a devision of the whole splitting into finitely many 
rectangles, one for each blur.  The homogeneous model $M$ in the second construction correspond to the second partition,  in the sense 
that it is {\it not}  the base of the representable term algebra,
but $W$ is, which is basically obtained by removing the blurs, that are the same time 
essential in representing it, $W$ thus corresponds to the term algebra of co-finite finite intersections with the second partition, 
which in turn is representable.

In short, we start up with a finite structure, blow it up, on the term algebra level, using blurs to represent it, 
but it will not be blurred enough to disappear on the complex algebra level, forcing
the latter to be non-representable (due to incompatibility of "a finitenes condition") with the inevitability 
of representing the complex algebra on an infinite set.

More formally, we define a function that maps the ingredients of the first construction to that of the second:

$$N\mapsto \bold M$$
$$N\times \omega\times n\mapsto \omega\times P\times J$$
In the former case $J'=\emptyset$, the blurs do not appear on this level, in the second splitting all blurs are used.
$$\{(\rho , i): i<n \}\mapsto \{ W: W\in J\}.$$
Here in the first case $n$ blurs are needed to represent the new term algebra. In the latter it is the number of two elements subsets of 
$I$.
For $\phi\in L^+$, let $\phi^M=\{s\in {}^nM: M\models \phi[s]\}.$
Here we are {\it not relativizing semantics}, in particular, tuples whose edges can be labelled $\rho$ are there,
but then the representable algebra is $\{\phi^M\cap W:\phi\in L^+\}.$
In the second case we have a finite partition of the rectangle $H$, via $(E^W: W\in J)$.
$$\{{\phi}^M\cap W: \phi \in L^+\}\mapsto \{X\subseteq F: X\cap E^W\in \Cof (E^W): \text { for all } W\in J\}.$$

$$\A\mapsto {\cal R}$$

$$\Cm\At\A\to \Cm\At{\cal R}$$

Here we include more examples.

\begin{example}

Let $l\in \omega$, $l\geq 2$, and let $\mu$ be a non-zero cardinal. Let $I$ be a finite set,
$|I|\geq 3l.$ Let
$$J=\{(X,n): X\subseteq I, |X|=l,n<\mu\}.$$
Let $H$ be as before, i.e.
$$H=\{a_i^{P,W}: i\in \omega, P\in I, W\in J\}.$$
Define
$(a_i^{P,S,p}, a_j^{Q,Z,q}, a_k^{R,W,r})$ is consistent ff

$S\cap Z\cap W=\emptyset$ or
$ e(i,j,k) \text { and } |\{P,Q,R\}|\neq 1.$

Pending on $l$ and $\mu$, let us call these atom structures ${\cal F}(l,\mu).$
Then our first  example in is just
${\cal F}(2,1).$

If $\mu\geq \omega$, then $J$ as defined above would be infinite,
and $\Uf$ will be a proper subset of the ultrafilters.
It is not difficult to show that if $l\geq \omega$
(and we relax the condition that $I$ be finite), then
$\Cm{\cal F}(l,\mu)$ is completely representable,
and if $l<\omega$ then $\Cm{\cal F}(l,\mu)$ is not representable. In the former case we have infnitely many colours, so that the chromatic number
of the graph is infinite, while in the second case the chromatic number is finite.

Informally, if the blurs get arbitarily large, then in the limit, the resulting algebra will be completely representable, and so its complex algebra
will be representable. If we take, a sequence of blurs, each finite, but increasing in size we get a sequence of algebras
that are not completely representable, and the sequence of their complex algebras will not be representable. The limit of the former,
will be  completely representable (with  an infinite set of blurs); its completion will be the limit of the second sequence of non representable algebras, 
will be representable. Either construction can be used to achieve this. 

This phenomena has many reincarnations in the literature. One is the following:
It is is nothing more than Monk's classical non finite
axiomatizability result;  it gives a sequence of non representable algebras whose ultraproduct is completely representable.
\end{example}
Using such examples, we now prove:
\begin{corollary}
\begin{enumarab}
\item  The classes $\RRA$ is not finitely axiomatizable.
\item  The elementary
closure of the class ${\CRA}$ is not finitely axiomatizable.
\end{enumarab}
\end{corollary}

\begin{demo}{Proof}
For the second  we use the second construction.  Let ${\cal D}$ be a non-
trivial ultraproduct of the atom structures ${\cal F}(i,1)$, $i\in \omega$. Then $\Cm{\cal D}$
is completely representable. 
Thus $\Tm{\cal F}(i,1)$ are $\RRA$'s 
without a complete representation while their ultraproduct has a complete representation.

Also $\Cm{\cal F}(i,1)$, $i\in \omega$ 
are non representable with a completely representable ultraproduct.
This yields the desired result.

We prove the cylindric case. Take $\G_i$ to be the disjoint union of cliques of size $n(n-1)/2+i$. 
Let $\alpha_i$ be the corresponding atom astructure
of $\A_i$, as constructed above. Then $\Cm \A_i$ is not representable, but $\prod_{i\in \omega}\Cm\A_i=\Cm(\prod_{i\in \omega}\A_i)$. 
Then the latter is based on the disjoint union of the cliques which is arbitrarily large, hence is representable. 
\end{demo}
The first construction also works, by using relation algebra atom structures with $n$ dimensional cylindric bases, this will
yield the analogous result for cylindric algebras.

The second re-incarnation is due to Hirsch and Hodkinson, it also works for relation and cylindric algebras, and this is the essence. For each graph
$\Gamma$, they associate a cylindric algebra atom structure of dimension $n$, $\M(\Gamma)$ such that $\Cm\M(\Gamma)$ is representable
if and only if the chomatic number of $\Gamma$, in symbols $\chi(\Gamma)$, which is the least number of colours needed, $\chi(\Gamma)$ is infinite. 
Using a famous theorem of Erdos, they construct  a sequence $\Gamma_r$ with infinite chromatic number and finite girth, 
whose limit is just $2$ colourable, they show that the class of strongly representable
algebras  is not elementary. Notice that this is a {\it reverse process} of Monk-like  constructions, given above, 
which gives a sequence of graphs of finite chromatic number whose limit (ultarproduct) has infinite
chromatic number.

And indeed, the construction also, 
is a reverse to Monk's construction in the following sense:
Some statement fail in $\A$ iff $At\A$ 
be partitioned into finitely many $\A$-definable sets with certain 
`bad' properties. Call this a {\it bad partition}. 
A bad partition of a graph is a finite colouring. So Monks result finds a sequence of badly partitioned atom structures,
converging to one that is not. As we did above, this boils down, to finding graphs of finite chromatic numbers $\Gamma_i$, having an ultraproduct
$\Gamma$ with infinite chromatic number. 

An atom structure is {\it strongly representable} iff it 
has {\it no bad partition using any sets at all}. So, here, the idea  find atom structures, with no bad partitions, 
with an ultraproduct that does have a bad partition.
From a graph Hirsch nad Hodkinson constructed  an atom structure that is strongly representable iff the graph
has no finite colouring.  So the problem that remains is to find a sequence of graphs with no finite colouring, 
with an ultraproduct that does have a finite colouring, that is, graphs of infinite chromatic numbers, having an ultraproduct
with finite chromatic number.

It is not obvious, a priori, that such graphs actually exist.
And here is where Erdos' methods offer solace. Indeed, graphs like this can be found using the probabilistic methods of Erdos, for those methods
render finite graphs of arbitrarily large chormatic number and girth. 
By taking disjoint unions, one {\it can get}  graphs of infinite chromatic number (no bad partitions) and arbitarly large girth. A non principal 
ultraproduct of these has no cycles, so has chromatic number 2 (bad partition).

Monk proved non finite axiomatizability of the representable algebras using a lifting argument. Here we do the same thing with anti-Monk algebras,
in the hope of getting a weakly representable $\omega$ dimensional atom structure that is not strongly representable.

Let $\Gamma_r$ be a sequence of Erdos graphs. Let $\A_n=\A(n,\Delta_n)$ be the representable atomic algebra of dimension $n$, based  on 
$\Delta_n$, the disjoint union of $\Gamma_r$, $r>n$. Then we know that $\At\A_n$ is a strongly representable atom structure of dimension $n$.
Let $\A_n^+$ be an $\omega$ dimensonal algebra such that 
$\Rd_n\A_{n}^+=\A_n$; we can assume that for $n<m$, there is an $x\in \A_m$, such that  $\A_{n}\cong \Rd\Rl_x\A_{m}$.
Now let $\A=\prod_{n\in F}\A_n^+$, be any non trivial ultarproduct of the $\A_n^+$s, 
then $\A$ is an atomic  $\RCA_{\omega}$ that has an atom structure that is based on the graph $\Delta=\prod \Delta_n$, 
with chromatic number $2$, hence it is only weakly representable.

\end{document}